\documentclass[12pt]{extarticle}
\usepackage{amssymb,amsfonts,amsthm,amsmath,bbold}
\usepackage{mathrsfs,pifont}
\usepackage{mathabx}
\usepackage[all]{xy}
\usepackage{array}
\usepackage{authblk}
 \usepackage{mathpple}
\usepackage{amssymb}
\usepackage{bm}
\usepackage{ytableau}
\usepackage{amsmath}
\usepackage{amsmath,amsfonts, amscd,amsthm, graphicx}
\usepackage{comment}
\usepackage[OT1]{fontenc} 
\usepackage[all]{xy}
\usepackage{color}
\usepackage{tikz-cd}
\usepackage{bm}
\usepackage{hyperref}

\usepackage[top=1.2in, bottom=1.2in, left=0.5in,
right=0.5in]{geometry}

\usepackage[font={small,sl}]{caption}
\usepackage{tikz}
\usepackage{tikz-cd}
\usetikzlibrary{matrix,arrows}

\usepackage{hyperref}
\usepackage[backrefs,msc-links]{amsrefs}

\newcounter{Chapcounter}

\newcommand{\chapter}[1] 
{ {\centering          
  \addtocounter{Chapcounter}{1} \Large \underline{\textbf{ \color{blue} Chapter \theChapcounter: ~#1}} }   
  \addcontentsline{toc}{section}{ \color{blue} Chapter:~\theChapcounter~~ #1}    
}

\theoremstyle{definition}

\theoremstyle{plain}
\newtheorem{thm}{Theorem}[section]

\newtheorem{thm-defn}{Theorem/Definition}[section]
\newtheorem{lem}[thm]{Lemma}
\newtheorem{lem-defn}[thm]{Lemma/Definition}
\newtheorem{prop}[thm]{Proposition}
\newtheorem{cor}[thm]{Corollary}
\newtheorem{conjecture}[thm]{Conjecture}
\newtheorem{prop-defn}[thm]{Proposition-Definition}

\newtheorem{defn}[thm]{Definition}

\newtheorem{thm-alg}[thm]{Theorem/Algorithm}

\newcommand{\mbf}{\mathbf}
\newcommand{\mbb}{\mathbb}
\newcommand{\mf}{\mathfrak}
\newcommand{\mc}{\mathcal}

\title{Quantum dynamical Weyl groups from quantum loop groups of arbitrary shuffle type}
\author{Tianqing Zhu}
\affil{Department of Mathematics, Columbia University \\
       tz2611@columbia.edu}

\begin{document}

\date{}

\maketitle
\begin{abstract} We construct quantum dynamical Weyl group elements associated with quantum loop groups of arbitrary shuffle type. Using the construction, we define the quantum lattice operator and the algebraic quantum difference equations for each tensor products of semisimple modules $V$ in category $\mathcal{O}$. We prove that algebraic quantum difference operators form a family of commuting operators, and they also commute with the qKZ operators for the tensor products of modules of the above type in $\mc{O}$. This recovers the construction in \cite{OS22} and can be viewed as the difference analog of the trigonmetric Casimir connection when the quantum loop group corresponds to the finite type symmetrisable Cartan matrix.
\end{abstract}

\tableofcontents

\section{\textbf{Introduction}}
The quantum dynamical Weyl group was first introduced systematically in \cite{EV02}. It is a dynamical analogue of the quantum Weyl group, defined for the quantum affine algebra $U_{q}(\hat{\mf{g}})$ rather than the finite-dimensional quantum group $U_{q}(\mf{g})$. A key ingredient of the construction is the dynamical monodromy operator $\mbf{B}_{w}(z)$ in Proposition 14 in \cite{EV02}.

A natural question is whether the construction can be generalised to the construction to more general settings. In \cite{OS22}, the generalisation of the dynamical monodromy operator was constructed via the stable envelopes for the equivariant $K$-theory of the Nakajima quiver varieties.

In this paper, under the setting of the quantum loop group $\mbf{U}_{Q}$ of the shuffle algebra type, we follow the approach of \cite{OS22} to construct two families of operators: dynamical monodromy operators $\mbf{B}_{\mbf{m}}(z)$ for $\mbf{m}\in\mbb{R}^{I}$ and the quantum lattice operators $\mbf{B}^s_{\mc{L}}(z)$ for $\mc{L}\in\mbb{Z}^{I}$. These operators are defined in equations \eqref{defn-of-monodromy} and \eqref{defn-of-quantum-lattice-operator}. The quantum lattice operator can be defined over arbitrary weight modules in $\mc{O}$ for the quantum loop group $\mbf{U}_{Q}$. In Theorem \ref{path-independence-theorem}, we prove that the quantum lattice operator $\mbf{B}_{\mc{L}}^s(z)$ is independent of the choice of the path from $s$ to $s-\mc{L}$ in $\mbb{R}^{I}$ with $\mc{L}\in\mbb{Z}^I$.

In the special case where $V\in\mc{O}$ is a semisimple module, or a tensor product of the semisimple modules. We can use the quantum lattice operator to construct the algebraic quantum difference equation. It takes the following form:
\begin{align}
\Psi(zp^{\mc{L}})=\mbf{M}_{\mc{L}}^s(z)\Psi(z),\qquad\mbf{M}_{\mc{L}}^s(z)=\mc{L}\prod_{\mbf{m}\in[s,s-\mc{L})}\mbf{B}_{\mbf{m}}(z)
\end{align}
where each dynamical monodromy operator $\mbf{B}_{\mbf{m}}(z)$ is defined in the slope subalgebra $\mc{B}_{\mbf{m}}$ in the quantum loop group $\mbf{U}_{Q}$, and $\mc{L}$ denotes the intertwining operator  associated to each $\mc{L}\in\mbb{Z}^I$, as explained in Lemma \ref{line-bundle-operator:lemma}. Moreover, Theorem \ref{holonomicity-qde:theorem} shows that the algebraic quantum difference operators $T_{\mc{L}}^{-1}\mbf{M}_{\mc{L}}^s(z)$ commute with each other for arbitrary $\mc{L}\in\mbb{Z}^I$.

Another important result is in Theorem \ref{qKZ-commutes-with-qde:theorem}, where we prove that the algebraic quantum difference equation commutes with the qKZ equation of the corresponding quantum loop group. This generalises the results of \cite{EV02}\cite{TV00} for quantum loop groups of the finite type quiver, as well as the result in \cite{OS22}.

\subsection{Relation to the Okounkov-Smirnov quantum difference equation}
One of the main motivations for this work is to provide the algebraic analogue of the quantum difference equations that arise in \cite{O15,OS22} from the $K$-theoretic quasimap counting for Nakajima quiver varieties.

In \cite{OS22}, Okounkov and Smirnov constructed the dynamical monodromy operator in the setting of the geometric quantum loop group defined by the $K$-theoretic stable envelopes. In particular, by \cite{Z25}, the geometric quantum loop group constructed in \cite{OS22} is isomorphic to the Drinfeld double of the preprojective $K$-theoretic Hall algebra, and it can be realised as a special case of the quantum loop group $\mbf{U}_{Q}$ in \cite{N22}. Thus, the geometric quantum difference equation corresponds to the algebraic quantum difference equation for the Drinfeld double of the preprojective $K$-theoretic Hall algebra in this paper.

An interesting problem for the algebraic quantum difference equations is the description for the connection matrix. In \cite{OS22}, the connection matrix for the geometric quantum difference equation is expressed in terms of elliptic stable envelopes \cite{AO21}\cite{O20}. For the algebraic quantum difference equation, it would be interesting to find an analogous description for the connection matrix of the algebraic quantum difference equations.

\subsection{Relation to the trigonometric Casimir connection}
The algebraic quantum difference equations can be thought of as the analogue of the trigonometric Casimir connections. The trigonometric Casimir connection was first introduced in \cite{L11}, it is a flat connection on a trivial bundle over the Cartan subalgebra, with values in a suitable completion of the Yangian $Y(\mathfrak{g}_{Q})$ of simple Lie algebra $\mf{g}_{Q}$. It is the trigonometric analogue for the rational Casimir connection with first-order singularities in terms of the Kahler variable $z=(z_1,\cdots,z_{I})\in(\mbb{C}^*)^I$. 

A generalisation of the trigonometric Casimir connection was introduced in \cite{MO12}, extending the construction to the Maulik-Okounkov Yangian $Y_{\hbar}(\mf{g}_{Q})$ for the Maulik-Okounkov Lie algebra $\mf{g}_{Q}$ of arbitrary quiver type, and it has been proved in \cite{BD23}\cite{SV23} that the Maulik-Okounkov Yangian is isomorphic to the double of the preprojective cohomological Hall algebra.

It has long been conjectured that as we take $e^{\kappa\hbar}$ and $\kappa\rightarrow0$, the quantum difference equation degenerates to the trigonometric Casimir connection. This has been proved in \cite{B15} that the conjecture is true for the quantum loop group with finite type symmetrisable Cartan matrix. It has also been proved in \cite{Z24}\cite{Z24-2} for the case of the quantum toroidal algebras to the affine Yangians.

In our setting, we expect that the algebraic quantum difference equation of arbitrary type quantum loop groups $\mbf{U}_{Q}$ of shuffle algebra type degenerates, as $\kappa\rightarrow0$,  to the trigonometric Casimir connection of the corresponding Yangian $\mbf{Y}_{Q}$ of the additive shuffle algebra type.

On the level of the monodromy representation for the trigonometric Casimir connection, we also expect the following description for the monodromy representation:
\begin{conjecture}
The monodromy representation for the trigonometric Casimir connection with the base point around $0$ is generated by the operator $\mbf{B}_{\mbf{m}}$ with $\mbf{B}_{\mbf{m}}:=\lim_{z\rightarrow\infty}\mbf{B}_{\mbf{m}}(z)$.
\end{conjecture}

This conjecture has been proved in the case of the quantum loop group $U_{q}(L\mf{sl}_2)$ in \cite{GL11}, and it was also proved in \cite{S24}\cite{Z24}\cite{Z24-2} for the quantum toroidal algebra $U_{q,t}(\hat{\hat{\mf{sl}}}_n)$. The study of such monodromy representations is an interesting problem. It would also be interesting to know how to describe the monodromy representation of different base points and on how to connect the monodromy representation for the different base points.

\subsection{Organisation of the paper}
The paper is organized as follows. In Section \ref{sec:_textbf_shuffle_algebra_of_arbitrary_quivers} we introduce the quantum loop group $\mbf{U}_{Q}$ of shuffle algebra type for arbitrary quivers. Section \ref{sec:_textbf_slope_factorization} discusses the slope factorization of the quantum loop group and recalls the new new coproduct introduced in \cite{N26}. In Section \ref{sec:_textbf_category_mc_o} we recall the definition of category $\mc{O}$ for $\mbf{U}_{Q}$.

Section \ref{sec:_textbf_dynamical_weyl_groups_for_shuffle_algebras} contains the main construction of the paper. We introduce the dynamical monodromy operator $\mbf{B}_{\mbf{m}}(z)$, the quantum lattice operator $\mbf{B}^s_{\mc{L}}(z)$ and the algebraic quantum difference equations for the quantum loop group $\mbf{U}_{Q}$. We prove the main result of the paper that the quantum lattice operator is independent of the choice of the path. We also prove that the algebraic quantum difference operator commutes with each other and commutes with the qKZ operator.

In Section \ref{sec:_textbf_examples}  we present two classes of examples: quantum loop groups of finite type with symmetrizable Cartan matrices, and the Drinfeld double of preprojective $K$-theoretic Hall algebras. In the first case, we prove that the generic dynamical monodromy operator is of $U_q(\mathfrak{sl}_2)$-type. In the second case, using the result of \cite{Z25}, we show that our construction reproduces the Okounkov-Smirnov quantum difference equation for Nakajima quiver varieties.

\subsection{Acknowledgements.}
The author is thankful to Andrei Okounkov, Andrey Smirnov for helpful discussion. The author is partially supported by the international collaboration grant BMSTC and ACZSP (Grant no. Z221100002722017) and by the National Key R \& D Program of
China (Grant no. 2020YFA0713000).

\section{\textbf{Shuffle algebra of arbitrary quivers}}\label{sec:_textbf_shuffle_algebra_of_arbitrary_quivers}
In this section we recall the definition of the shuffle algebra for an arbitrary quiver. For further details, the reader is referred to \cite{N22-2, N26}.

We fix a quiver $Q=(I,E)$, and denote by $\#_{ij}$ the number of arrows from $i$ to $j$.

We define a fraction field $\mbb{K}$ which contains the rational number field $\mbb{Q}$. In this paper we consider a collection of rational functions:
\begin{align*}
\zeta_{ij}(x)=\frac{\tilde{\zeta}_{ij}(x)}{(1-x)^{\delta_{ij}}}
\end{align*}
where
\begin{align*}
\tilde{\zeta}_{ij}(x)=c_{ij}x^{\#_{ij}+\delta_{ij}}+\cdots+c_{ij}'x^{-\#_{ji}}
\end{align*}
for scalars $c_{ij},c_{ij}'\in\mbb{K}^*$ and integers $\#_{ij}$ such that $c_{ij}x^{\#_{ij}}$ is the leading order term of $\zeta_{ij}(x)$ as $x\rightarrow\infty$. We define the bilinear form:
\begin{align}\label{bilinear-form}
\langle\mbf{m},\mbf{n}\rangle=\sum_{i,j\in I}m_in_j\#_{ij},\qquad\mbf{m},\mbf{n}\in\mbb{N}^I.
\end{align}

We also define the following constants which will be used throughout the paper:
\begin{align}\label{important-constant}
\gamma_{\mbf{m},\mbf{n}}=\prod_{i,j\in I}[(-1)^{\delta_{ij}}\frac{c_{ij}}{c_{ji}'}]^{m_in_j}\in\mbb{K}^*.
\end{align}

We also assume that $\lim_{x\rightarrow\infty}\frac{\zeta_{ij}(x)}{\zeta_{ji}(x^{-1})}<\infty$.

\subsection{Prequantum loop group}
We follow the definition given in \cite{N22-2}.

\begin{defn}
The positive part of the \textbf{prequantum loop group} $\tilde{\mbf{U}}^+_{Q}$ associated to the datum $\{\zeta_{ij}\}_{i,j\in I}$ is the $\mbb{K}$-algebra generated by symbols
\begin{align*}
\{e_{i,d}\}_{i\in I,d\in\mbb{Z}}
\end{align*}
modulo the relations for all $i,j\in I$
\begin{align*}
e_i(z)e_j(w)\zeta_{ji}(\frac{w}{z})=e_{j}(w)e_{i}(z)\zeta_{ij}(\frac{z}{w}),\qquad e_i(z)=\sum_{d\in\mbb{Z}}\frac{e_{i,d}}{z^d}.
\end{align*}
\end{defn}

The algebra $\tilde{\mbf{U}}^+_{Q}$ is an $\mbb{N}^I\times\mbb{Z}$-graded $\mbb{K}$-vector space with the grading given by:
\begin{align*}
\text{deg }e_{i,d}=(\mbf{e}_i,d)
\end{align*}
with $\mbf{e}_i$ the $i$-th unit vector in $\mbb{N}^I$.
As a $\mbb{K}$-graded vector space, $\tilde{\mbf{U}}^+_{Q}$ decomposes into the following $(\mbb{N}^I\times\mbb{Z})$-graded components:
\begin{align}\label{graded-pieces-of-prequantum-loop-group}
\tilde{\mbf{U}}^+_{Q}=\bigoplus_{\mbf{n}\in\mbb{N}^I}\tilde{\mbf{U}}^+_{Q,\mbf{n}}=\bigoplus_{(\mbf{n},d)\in\mbb{N}^I\times\mbb{Z}}\tilde{\mbf{U}}^+_{Q,\mbf{n},d},\qquad\tilde{\mbf{U}}_{Q,\mbf{n}}^+:=\bigoplus_{d\in\mbb{Z}}\tilde{\mbf{U}}_{Q,\mbf{n},d}^+.
\end{align}
The positive half $\tilde{\mbf{U}}^+_{Q}$ carries a shift automorphism
\begin{equation}\label{shift-automorphism-quantum-positive}
\begin{tikzcd}
\tilde{\mbf{U}}^+_{Q}\arrow[r,"\tau_{\mbf{k}}"]&\tilde{\mbf{U}}^+_{Q},\qquad e_{i,d}\mapsto e_{i,d+k_i}
\end{tikzcd}
\end{equation}
for any integral vector $\mbf{k}=(k_i)_{i\in I}\in\mbb{Z}^I$.

The negative part of the prequantum loop group is defined as the opposite algebra:
\begin{align*}
\tilde{\mbf{U}}^-_{Q}=\tilde{\mbf{U}}_{Q}^{+,op}
\end{align*}
with generators denoted by $f_{i,d}$ in place of $e_{i,d}$ satisfying the relations:
\begin{align*}
f_i(z)f_j(w)\zeta_{ij}(\frac{z}{w})=f_j(w)f_i(z)\zeta_{ji}(\frac{w}{z}),\qquad f_{i}(z)=\sum_{d\in\mbb{Z}}\frac{f_{i,d}}{z^d}.
\end{align*}

It also carries a shift automorphism $\tau_{\mbf{k}}$ given by:
\begin{equation}\label{shift-automorphism-quantum-negative}
\begin{tikzcd}
\tilde{\mbf{U}}_{Q}^{-}\arrow[r,"\tau_{\mbf{k}}"]&\tilde{\mbf{U}}_{Q}^{-}
\end{tikzcd},\qquad f_{i,d}\mapsto f_{i,d-k_i}
\end{equation}
for $\mbf{k}=(k_i)_{i\in I}\in\mbb{Z}^I$.

\subsection{Shuffle algebras}

We consider an infinite collection of variables $z_{i1},z_{i2},\cdots$ for all $i\in I$. For any $\mbf{n}=(n_i)_{i\in I}\in\mbb{N}^I$, we write $|\mbf{n}|=\sum_{i\in I}n_i$ and $\mbf{n}!=\prod_{i\in I}n_i!$. 

\begin{defn}
The \textbf{big shuffle algebra} associated with the datum of rational functions $\{\zeta_{ij}\}_{i,j\in I}$ is 
\begin{align*}
\mc{V}^{+}_{Q}=\bigoplus_{\mbf{n}\in\mbb{N}^I}\mc{V}_{Q,\mbf{n}}^+,\qquad\mc{V}_{Q,\mbf{n}}^+:=\mbb{K}[z_{i1}^{\pm1},\cdots,z_{in_i}^{\pm1}]_{i\in I}^{\text{Sym}}
\end{align*}
where "Sym" denotes color-symmetric Laurent polynomials , i.e. it is symmetric in the variables $z_{i1},\cdots,z_{in_i}$ for each $i\in I$ separately. The vector space is equipped with the shuffle multiplication:
\begin{equation*}
\begin{aligned}
&F(\cdots,z_{i1},\cdots,z_{in_i},\cdots)*G(\cdots,z_{i1},\cdots,z_{in_i'},\cdots)=\\
&\text{Sym}[\frac{F(\cdots,z_{i1},\cdots,z_{in_i},\cdots)G(\cdots,z_{i,n_i+1},\cdots,z_{i,n_i+n_i'})}{\mbf{n}!\mbf{n}'!}\prod_{\substack{1\leq a\leq n_i\\n_j<b\leq n_j+n_j'}}^{i,j\in I}\zeta_{ij}(\frac{z_{ia}}{z_{jb}})]
\end{aligned}
\end{equation*}
and here "Sym" means the symmetrisation with repsect to $(\mbf{n}+\mbf{n}')!:=\prod_{i\in I}(n_i+n_i')!$ permutations of the variables $\{z_{i1},\cdots,z_{i,n_i+n_i'}\}$ for each $i$ independently.
\end{defn}

The algebra $\mc{V}^+_{Q}$ is graded by $\mbb{N}^I\times\mbb{Z}$, where the $\mathbb{N}^I$-degree records the number of variables of each color and the $\mathbb{Z}$-degree is the homogeneous degree. For $R(\cdots,z_{i1},\cdots,z_{in_i},\cdots)\in\mc{V}^+_{Q}$, we set:
\begin{align*}
\text{deg }R(\cdots,z_{i1},\cdots,z_{in_i},\cdots)=(\mbf{n},\text{hom deg }R)
\end{align*}
where $\text{hom deg }R\in\mbb{Z}$ is the homogeneous degree of the colored-symmetric Laurent polynomial $R$.

The big shuffle algebra $\mc{V}_{Q}^+$ has the $\mbb{N}^I\times\mbb{Z}$-grading as:
\begin{align*}
\mc{V}^+_{Q}=\bigoplus_{\mbf{n}\in\mbb{N}^I}\mc{V}_{Q,\mbf{n}}^+=\bigoplus_{(\mbf{n},d)\in\mbb{N}^I\times\mbb{Z}}\mc{V}_{Q,\mbf{n},d}^+.
\end{align*}

There is also a shift automorphism given by:
\begin{equation}\label{shift-automorphism-shuffle-big}
\begin{tikzcd}
\mc{V}^+_{Q}\arrow[r,"\tau_{\mbf{k}}"]&\mc{V}^+_{Q}
\end{tikzcd},\qquad R(\cdots,z_{ia},\cdots)\mapsto R(\cdots,z_{ia},\cdots)\prod_{i\in I,a\geq1}z_{ia}^{k_i}
\end{equation}
for any $\mbf{k}=(k_i)_{i\in I}\in\mbb{Z}^I$.
Similarly we define the negative half of the shuffle algebra as $\mc{V}^-_Q=\mc{V}^{+,op}_{Q}$. It also has the shift automorphism given by:
\begin{equation}\label{negative-shift-automorphism-shuffle-big}
\begin{tikzcd}
\mc{V}^-_{Q}\arrow[r,"\tau_{\mbf{k}}"]&\mc{V}^-_{Q}
\end{tikzcd},\qquad R(\cdots,z_{ia},\cdots)\mapsto R(\cdots,z_{ia},\cdots)\prod_{i\in I,a\geq1}z_{ia}^{-k_i}.
\end{equation}

To relate $\tilde{\mbf{U}}_{Q}$ with the big shuffle algebra, one can observe that there is a homomorphism of $\mbb{N}^I\times\mbb{Z}$ graded $\mbb{K}$-algebras
\begin{align}\label{algebra-map-core}
i^{\pm}:\tilde{\mbf{U}}^{\pm}_{Q}\rightarrow\mc{V}^{\pm}_{Q},\qquad e_{i,d}, f_{i,d}\mapsto z_{i1}^d\in\mc{V}_{Q}^{\pm}.
\end{align}

It is obvious that the image 
\begin{align*}
\mc{S}^{\pm}_{Q}:=\text{Im}(i^{\pm})
\end{align*}
is generated by $\{z_{i1}^d\}$. The kernel $K^{\pm}=\text{Ker }i^{\pm}$ is a two-sided ideal of $\tilde{\mbf{U}}^+_{Q}$. 

\begin{defn}
We define the positive/negative part of the quantum loop group associated to the datum $\{\zeta_{ij}\}_{i,j\in I}$ as the $\mbb{K}$-algebra
\begin{align*}
\mbf{U}^{\pm}_{Q}:=\tilde{\mbf{U}}^{\pm}_{Q}/K^{\pm}
\end{align*}
which induces the isomorphisms:
\begin{align}\label{isom-of-shuffle-and-quotient}
i^{\pm}:\mbf{U}_{Q}^{\pm}\cong\mc{S}^{\pm}.
\end{align}
\end{defn}

\begin{lem}\label{shifted-automorphism-on-quantum-loop-group:lemma}
The shift automorphism described in \eqref{shift-automorphism-shuffle-big} and \eqref{negative-shift-automorphism-shuffle-big} can be restricted to the shift automorphism $\tau_{\mbf{k}}:\mc{S}^{\pm}\rightarrow\mc{S}^{\pm}$ for $\mbf{k}\in\mbb{Z}^I$. Moreover, via the isomorphism \eqref{isom-of-shuffle-and-quotient} induces the shift automorphism $\tau_{\mbf{k}}:\mbf{U}_{Q}^{\pm}\rightarrow\mbf{U}_{Q}^{\pm}$ for $\mbf{k}\in\mbb{Z}^I$.
\end{lem}
\begin{proof}
We give the proof for $\mc{S}^+$ since the proof for $\mc{S}^-$ is similar. The second part of the lemma follows immediately from \eqref{isom-of-shuffle-and-quotient} and the first part of the lemma.

Since $\mc{S}^+$ is generated by $\{z_{i1}^d\}_{i\in I}$, one can write down the elements in $\mc{S}^+$ as:
\begin{align*}
\sum_{(\mbf{i},\mbf{d})}a_{\mbf{i}}z_{i_1,1}^{d_1}*\cdots*z_{i_n,1}^{d_n}.
\end{align*}
The image of the above term under the shifted automorphism $\tau_{\mbf{k}}$ is given by:
\begin{align}\label{term-after-shifted}
(\sum_{(\mbf{i},\mbf{d})}a_{\mbf{i}}z_{i_1,1}^{d_1}*\cdots*z_{i_n,1}^{d_n})\prod_{\substack{i\in I\\a\geq0}}z_{ia}^{k_i}
\end{align}
and thus it remains to prove that \eqref{term-after-shifted} is in $\mc{S}^+$. Since $\tau_{\mbf{k}}$ is the algebra automorphism, we have that:
\begin{align*}
\tau_{\mbf{k}}(\sum_{(\mbf{i},\mbf{d})}a_{\mbf{i}}z_{i_1,1}^{d_1}*\cdots*z_{i_n,1}^{d_n})=\sum_{(\mbf{i},\mbf{d})}a_{\mbf{i}}(\tau_{\mbf{k}}z_{i_1,1}^{d_1})*\cdots*(\tau_{\mbf{k}}z_{i_n,1}^{d_n})
\end{align*}
and since $\tau_{\mbf{k}}z_{i1}^d=z_{i1}^{d+k_i}$ is in $\mc{S}^+$, the proof is finished.
\end{proof}

\subsection{Quantum loop groups}

We consider the Drinfeld version extended algebras:
\begin{align*}
\tilde{\mbf{U}}^{\geq}_{Q}=\frac{\tilde{\mbf{U}}^+_{Q}\otimes\mbb{K}[h_{i,0}^{+},h_{i,1}^+,h_{i,2}^+,\cdots]_{i\in I}}{h_i^+(z)e_j(w)\zeta_{ji}(\frac{w}{z})=e_{j}(w)h_i^+(z)\zeta_{ij}(\frac{z}{w})}
\end{align*}

\begin{align*}
\tilde{\mbf{U}}^{\leq}_{Q}=\frac{\tilde{\mbf{U}}^-_{Q}\otimes\mbb{K}[h_{i,0}^{-},h_{i,-1}^{-},h_{i,-2}^-,\cdots]_{i\in I}}{h_{i}^-(z)f_j(w)\zeta_{ij}(\frac{z}{w})=f_j(w)h_i^-(z)\zeta_{ji}(\frac{w}{z})}
\end{align*}
where
\begin{align*}
h_{i}^+(z)=\sum_{d=0}^{\infty}\frac{h_{i,d}^+}{z^d},\qquad h_{i}^-(z)=\sum_{d=0}^{\infty}h_{i,-d}^{-}z^d
\end{align*}

Since the assumption implies that $\frac{\zeta_{ij}(x)}{\zeta_{ji}(x^{-1})}$ is regular and nonzero at both $x=0$ and $x=\infty$. We can similarly define:
\begin{align*}
\mc{V}_{Q}^{\geq}=\mc{V}^+_{Q}\otimes\mbb{K}[h_{i,d}^+]_{i\in I,d\geq0},\qquad\mc{V}_{Q}^{\leq}=\mc{V}^-_{Q}\otimes\mbb{K}[h_{i,d}^-]_{i\in I,d\leq0}
\end{align*}
with the following relations:
\begin{equation*}
h_i^+(y)F(z_{j1},\cdots,z_{jn_j})_{j\in I}=F(z_{j1},\cdots,z_{jn_j})_{j\in I}h_i^+(y)\prod_{j\in I}\prod_{a=1}^{n_j}\frac{\zeta_{ij}(\frac{y}{z_{ja}})}{\zeta_{ji}(\frac{z_{ja}}{y})}
\end{equation*}

\begin{equation*}
G(z_{j1},\cdots,z_{jn_j})_{j\in I}h_{i}^-(y)=h_i^-(y)G(z_{j1},\cdots,z_{jn_j})_{j\in I}\prod_{j\in I}\prod_{a=1}^{n_j}\frac{\zeta_{ij}(\frac{z_{ja}}{y})}{\zeta_{ji}(\frac{z_{ja}}{y})}
\end{equation*}
for $F\in\mc{V}_{Q}^+$ and $G\in\mc{V}_{Q}^-$.

Since the homomorphism $i^{\pm}$ repsect all the constructions above, we obtain the algebra structures on:
\begin{align*}
\mbf{U}_{Q}^{\geq}=\mbf{U}^+_{Q}\otimes\mbb{K}[h_{i,d}^+]_{i\in I,d\geq0},\qquad\mbf{U}_{Q}^{\leq}=\mbf{U}^-_{Q}\otimes\mbb{K}[h_{i,d}^-]_{i\in I,d\leq0}.
\end{align*}

More generally, given elements $F\in\mbf{U}^+_{Q,\mbf{n}}$, $G\in\mbf{U}^-_{Q,-\mbf{n}}$, we have that:
\begin{align*}
Fh_j^{\pm}(w)=h_{j}^{\pm}(w)F\prod_{1\leq a\leq n_i}^{i\in I}\frac{\zeta_{ij}(\frac{z_{ia}}{w})}{\zeta_{ji}(\frac{w}{z_{ia}})}\\
Gh_{j}^{\pm}(w)=h_{j}^{\pm}(w)G\prod_{1\leq a\leq n_i}^{i\in I}\frac{\zeta_{ji}(\frac{w}{z_{ia}})}{\zeta_{ij}(\frac{z_{ia}}{w})}.
\end{align*}

For simplicity, we denote the following completion for the tensor product of the $\mbb{Z}^I\times\mbb{Z}$-algebra $A$:
\begin{align}\label{V-and-H-coproduct}
A\otimes^{V}A=\bigoplus_{(\mbf{n},d)\in\mbb{Z}^I\times\mbb{Z}}(A\otimes^VA)_{\mbf{n},d},\qquad A\otimes^HA=\bigoplus_{(\mbf{n},d)\in\mbb{Z}^I\times\mbb{Z}}(A\otimes^HA)_{\mbf{n},d}
\end{align}
where $(A\otimes^VA)_{\mbf{n},d}$ consists of infinite sums of tensors $x\otimes y$ where all but finitely many summands have the property that $\text{vdeg }x\geq N$ and $\text{vdeg }y\leq -N$. $(A\otimes^HA)_{\mbf{n},d}$ consists of infinite sums of tensors $x\otimes y$ where all but finitely many summands have that $|\text{hdeg }x|\leq-N$ and $|\text{hdeg }y|\geq N$ for any $N\in\mbb{N}$.

\subsubsection{The Drinfeld coproduct}
The algebra $\tilde{\mbf{U}}_{Q}$ has a Drinfeld type coproduct $\Delta:\tilde{\mbf{U}}_{Q}^{\geq}\rightarrow\tilde{\mbf{U}}_{Q}^{\geq}\otimes^V\tilde{\mbf{U}}_{Q}^{\geq}$ and $\Delta:\tilde{\mbf{U}}_{Q}^{\leq}\rightarrow\tilde{\mbf{U}}_{Q}^{\leq}\otimes^V\tilde{\mbf{U}}_{Q}^{\leq}$ defined by:
\begin{equation*}
\begin{aligned}
&\Delta(h_i^{\pm}(z))=h_i^{\pm}(z)\otimes h_i^{\pm}(z)\\
&\Delta(e_i(z))=h_i^+(z)\otimes e_i(z)+e_i(z)\otimes1\\
&\Delta(f_i(z))=1\otimes f_i(z)+f_i(z)\otimes h_i^-(z)
\end{aligned}
\end{equation*}

Similarly on the side of the shuffle algebra, we have the Drinfeld type copoduct:
\begin{align}\label{Drinfeld-coproduct-definition}
\Delta:\mc{V}_{Q}^{\geq}\rightarrow\mc{V}_{Q}^{\geq}\otimes^V\mc{V}_{Q}^{\geq},\qquad\Delta:\mc{V}_{Q}^{\leq}\rightarrow\mc{V}_{Q}^{\leq}\otimes^V\mc{V}_{Q}^{\leq}
\end{align}
written as:
\begin{align*}
\Delta(F)=\sum_{[0\leq k_i\leq n_i]_{i\in I}}\frac{\prod_{k_j<b\leq n_j}^{j\in I}h_{j}^{+}(z_{jb})F(\cdots,z_{i1},\cdots,z_{ik_i}\otimes z_{i,k_i+1},\cdots,z_{in_i},\cdots)}{\prod_{1\leq a\leq k_i}^{i\in I}\prod_{k_j<b\leq n_j}^{j\in I}\zeta_{ji}(z_{jb}/z_{ia})}
\end{align*}

\begin{align*}
\Delta(G)=\sum_{[0\leq k_i\leq n_i]_{i\in I}}\frac{G(\cdots,z_{i1},\cdots,z_{ik_i}\otimes z_{i,k_i+1},\cdots,z_{in_i},\cdots)\prod_{1\leq b\leq k_j}^{j\in I}h_{j}^{-}(z_{jb})}{\prod_{1\leq a\leq k_i}^{i\in I}\prod_{k_j<b\leq n_j}^{j\in I}\zeta_{ij}(z_{ia}/z_{jb})}
\end{align*}
for $F\in\mc{V}^{\geq}_{Q,\mbf{n}}$ and $G\in\mc{V}^{\leq}_{Q,-\mbf{n}}$.

It is easy to see that the algebra in \eqref{algebra-map-core} respects the coproduct structures on both sides. Thus it descends to the coproduct on
\begin{align*}
\Delta:\mbf{U}_{Q}^{\geq}\rightarrow\mbf{U}_{Q}^{\geq}\otimes^{V}\mbf{U}_{Q}^{\geq},\qquad\Delta:\mbf{U}_{Q}^{\leq}\rightarrow\mbf{U}_{Q}^{\leq}\otimes^{V}\mbf{U}_{Q}^{\leq}
\end{align*}

\subsubsection{Bialgebra pairing}
We use the following notation for the rational function $G$:
\begin{align*}
\int_{|z_1|>>\cdots>>|z_n|}G(z_1,\cdots,z_n)
\end{align*}
denotes the constant term in the expansion of $G$ as a power series in
\begin{align*}
\frac{z_2}{z_1},\cdots,\frac{z_n}{z_{n-1}}
\end{align*}
and we also define
\begin{align*}
\int_{|z_1|<<\cdots<<|z_n|}G(z_1,\cdots,z_n)
\end{align*}
in a similar way.

\begin{defn}\label{bialgebra-pairing-for-the-large-algebra:label}
There is a bialgebra pairing:
\begin{equation*}
\langle-,-\rangle:\tilde{\mbf{U}}_{Q}^+\otimes\mc{V}_{Q}^-\rightarrow\mbb{K},\qquad\langle-,-\rangle:\mc{V}_Q^+\otimes\tilde{\mbf{U}}_Q^-\rightarrow\mbb{K}
\end{equation*}
given for all $E\in\mc{V}_{Q,\mbf{n}}^+$, $F\in\mc{V}_{Q,-\mbf{n}}^-$ and all $i_1,\cdots,i_n\in I$, $d_1,\cdots,d_n\in\mbb{Z}$ by:
\begin{align*}
\langle e_{i_1,d_1}\cdots e_{i_n,d_n},F\rangle=\int_{|z_1|>>\cdots>>|z_n|}\frac{z_1^{d_1}\cdots z_n^{d_n}F(z_1,\cdots,z_n)}{\prod_{1\leq a<b\leq n}\zeta_{i_bi_a}(\frac{z_b}{z_a})}
\end{align*}
\begin{align*}
\langle E,f_{i_1,d_1}\cdots f_{i_n,d_n}\rangle=\int_{|z_1|<<\cdots<<|z_n|}\frac{z_1^{d_1}\cdots z_n^{d_n}E(z_1,\cdots,z_n)}{\prod_{1\leq a<b\leq n}\zeta_{i_ai_b}(\frac{z_a}{z_b})}
\end{align*}
if $\mbf{e}_{i_1}+\cdots+\mbf{e}_{i_n}=\mbf{n}$, and $0$ otherwise.
\end{defn}

It was shown in \cite{N22-2} that the above bialgebra pairing can descend to the following bialgebra pairing
\begin{align*}
\langle-,-\rangle:\mbf{U}_{Q}^+\otimes\mc{S}_{Q}^-\rightarrow\mbb{K},\qquad\langle-,-\rangle:\mc{S}_Q^+\otimes\mbf{U}_{Q}^-\rightarrow\mbb{K}
\end{align*}
which is nondegenerate as being proved in \cite{N22-2}. Moreover, via the isomorphism \eqref{isom-of-shuffle-and-quotient}, 
we can have the nondegenerate bialgebra pairing:
\begin{align}\label{nondegenerate-bialgebra-pairing}
\langle-,-\rangle:\mbf{U}^{\geq}_{Q}\otimes\mbf{U}^{\leq}_{Q}\rightarrow\mbb{K}
\end{align}
which is defined as:
\begin{align*}
\langle h_{i}^+(z),h_{j}^-(w)\rangle=\frac{\zeta_{ij}(\frac{z}{w})}{\zeta_{ji}(\frac{w}{z})},\qquad\text{expanded as }|z|>>|w|
\end{align*}

\begin{align*}
\langle h_{i,0}^+,h_{j,0}^-\rangle=\gamma_{\mbf{e}_i,\mbf{e}_j}
\end{align*}

\begin{align*}
\langle e_{i,d},f_{j,k}\rangle=\delta_{ij}\delta_{d+k,0}
\end{align*}
which is a perfect pairing.

At the stage we can introduce the definition of the quantum loop group that is used in this paper.

\begin{defn}
The \textbf{quantum loop group} is defined as
\begin{align*}
\mbf{U}_{Q}=\mbf{U}^{\geq}_{Q}\otimes\mbf{U}^{\leq}_{Q}/(h_{i,0}^+\otimes h_{i,0}^-=1)
\end{align*}
with the multiplication governed by the Drinfeld double relation
\begin{align*}
(x\otimes y)(x'\otimes y')=\sum_{a}\sum_{b}\langle S(x_{1,a}'),y_{1,b}\rangle xx_{2,a}'\otimes y_{2,b}y'\langle x_{3,a}',y_{3,b}\rangle.
\end{align*}
\end{defn}

Using the relation, it is straightforward to deduce the relations:
\begin{align*}
[e_i(z),f_j(w)]=\delta_{ij}\delta(\frac{z}{w})\cdot[h_{i}^+(z)-h_{i}^-(w)].
\end{align*}

Similarly, we can define the quadratic quantum loop group $\tilde{\mbf{U}}_{Q}$ as the algebra generated by $\{e_{i,d},f_{i,d},h_{i,n},h_{i,-n}'\}_{i\in I,d\in\mbb{Z},n\geq0}$ subject to the following relations:
\begin{align*}
e_i(z)e_j(w)\zeta_{ji}(\frac{w}{z})=e_j(w)e_i(z)\zeta_{ij}(\frac{z}{w}),\qquad f_i(z)f_j(w)\zeta_{ij}(\frac{z}{w})=f_j(w)f_i(z)\zeta_{ji}(\frac{w}{z})
\end{align*}

\begin{align*}
h_i^+(z)e_j(w)\zeta_{ji}(\frac{w}{z})=e_{j}(w)h_i^+(z)\zeta_{ij}(\frac{z}{w}),\qquad h_{i}^-(z)f_j(w)\zeta_{ij}(\frac{z}{w})=f_j(w)h_i^-(z)\zeta_{ji}(\frac{w}{z})
\end{align*}

The Drinfeld double $\mbf{U}_{Q}$ with respect to the topological coproduct structure $\Delta$ will give us a universal $R$-matrix $\mc{R}\in\mbf{U}_{Q}^{\geq}\hat{\otimes}\mbf{U}_{Q}^{\leq}$.

\subsection{Examples for the quantum loop group}
In this subsection we illustrate some examples of the quantum loop group that might be familiar to the readers who work on the representation theory of quantum affine algebras.
\subsubsection{Quantum loop group associated to finite type symmetrisable Cartan matrix}\label{ssub:quantum_loop_group_associated_to_finite_type_symmetrisable_cartan_matrix}
Now we take $\mbb{K}=\mbb{Q}(q)$. If we fix a finite set $I$ and a symmetrisable Cartan matrix
\begin{align*}
(\frac{2d_{ij}}{d_{ii}}\in\mbb{Z})_{i,j\in I}
\end{align*}
where $d_ii$ are even positive integers with $\text{gcd }2$, and $d_{ij}=d_{ji}$ are non-positive integers. We assume that the symmetrisable Carten matrix is of the finite type, which means that it will determine a finite type Kac-Moody Lie algebra $\mf{g}$.

For the quantum loop group $\mbf{U}_{Q}:=U_{q}(L\mf{g})$ associated with the above symmetrisable Cartan matrix, it will be associated with the data $\{\zeta_{ij}(x)\}_{i,j\in I}$ chosen as:
\begin{align*}
\zeta_{ij}(x)=\frac{x-q^{-d_{ij}}}{x-1}
\end{align*}

In this case this is the quantum loop group that has been studied by large amounts of the literature (e.g.). In this case the kernel $K^{+}$ of the map \eqref{algebra-map-core} is described by the Drinfeld's $q$-Serre relations for all $i\neq j$ in $I$:
\begin{align*}
\sum_{\sigma\in S_{n}}\sum_{k=0}^n(-1)^k\binom{n}{k}_ie_{i}(x_{\sigma(1)})\cdots e_{i}(x_{\sigma(k)})e_j(y)e_i(x_{\sigma(k+1)})\cdots e_{i}(x_{\sigma(n)})=0
\end{align*}
where $n=1-\frac{2d_{ij}}{d_{ii}}$, $\binom{n}{k}_i=\frac{[n]_i!}{[k]_i![n-k]_i!}$ with $[k]_i!=[1]_i\cdots[k]_i$ and $[k]_i=\frac{q_i^{k}-q_i^{-k}}{q_i-q_i^{-1}}$.

\subsubsection{Drinfeld double of the preprojective $K$-theoretic Hall algebra}
Fix the quiver $Q=(I,E)$. Now we take $\mbb{K}=\mbb{Q}(q,t_e)_{e\in E}$. The localised preprojective $K$-theoretic Hall algebra $\mc{A}_{Q}^+$ was defined in \cite{VV22}. It has the positive half of the quantum loop group $\mbf{U}_{Q}^+$ description with the datum of rational functions $\{\zeta_{ij}(x)\}_{i,j\in I}$ given by \cite{N22, N23, N23-2}:
\begin{align*}
\zeta_{ij}(x)=(\frac{1-q^{-1}x}{1-x})^{\delta^i_j}\prod_{e=ij\in E}(1-t_ex)\prod_{e=ji}(1-\frac{q}{t_ex}).
\end{align*}

Then the Drinfeld double of the preprojective $K$-theoretic Hall algebra $\mc{A}_{Q}$ is the quantum loop group $\mbf{U}_Q$ associated with the above datum of rational functions. It was proved in \cite{Z25} that $\mc{A}_{Q}$ is isomorphic to the geometric quantum loop group $U_{q}^{MO}(\hat{\mf{g}}_{Q})$ defined in \cite{OS22}.

\section{\textbf{Slope factorization}}\label{sec:_textbf_slope_factorization}

In this subsection we introduce the slope factorisation and slope subalgebras for the quantum loop group $\mbf{U}_{Q}$. In this paper we take $\mbf{U}_{Q}^{\pm}$ as the shuffle algebras $\mc{S}_Q^{\pm}$ via the isomorphism \eqref{isom-of-shuffle-and-quotient}. We always consider the elements in $\mbf{U}_{Q}^{\pm}$ as the shuffle elements in $\mc{S}_Q^{\pm}$.

For the detailed reference for the slope subalgebra one can refer to \cite{N22}\cite{N22-2}\cite{N26}.

\begin{defn}
For $F\in\mbf{U}^{+}_{Q}$, we say that $F$ is of slope $\leq\mbf{m}$ if:
\begin{align*}
\lim_{\xi\rightarrow\infty}\frac{F(\cdots,\xi z_{i1},\cdots,\xi z_{i,k_i},z_{i,k_i+1},\cdots,z_{i,n_i},\cdots)}{\xi^{\mbf{m}\cdot\mbf{k}+\langle\mbf{k},\mbf{n}-\mbf{k}\rangle}}
\end{align*}
is finite for all $\mbf{0}\leq\mbf{k}\leq\mbf{n}$.

We also say that $F$ is of slope $\geq\mbf{m}$ if:
\begin{align*}
\lim_{\xi\rightarrow0}\frac{F(\cdots,\xi z_{i1},\cdots,\xi z_{ik_i},z_{i,k_i+1},\cdots,z_{i,n_i},\cdots)}{\xi^{\mbf{m}\cdot\mbf{k}-\langle\mbf{n}-\mbf{k},\mbf{k}\rangle}}
\end{align*}
is finite for all $\mbf{0}\leq\mbf{k}\leq\mbf{n}$.

Similarly, we will say that $G\in\mbf{U}^{-}_{Q}$ has slope $\leq\mbf{m}$ if:
\begin{align*}
\lim_{\xi\rightarrow0}\frac{G(\cdots,\xi z_{i1},\cdots,\xi z_{i,k_i},z_{i,k_i+1},\cdots,z_{i,n_i},\cdots)}{\xi^{-\mbf{m}\cdot\mbf{k}-\langle\mbf{n}-\mbf{k},\mbf{k}\rangle}}
\end{align*}
is finite for all $\mbf{0}\leq\mbf{k}\leq\mbf{n}$.

We also say that $G$ is of slope $\geq\mbf{m}$ if:
\begin{align*}
\lim_{\xi\rightarrow\infty}\frac{G(\cdots,\xi z_{i1},\cdots,\xi z_{i,k_i},z_{i,k_i+1},\cdots,z_{i,n_i},\cdots)}{\xi^{-\mbf{m}\cdot\mbf{k}+\langle\mbf{k},\mbf{n}-\mbf{k}\rangle}}
\end{align*}
is finite for all $\mbf{0}\leq\mbf{k}\leq\mbf{n}$.
\end{defn}

Let us denote the subspaces of shuffle elements of slope $\leq\mbf{m}$ and $\geq\mbf{m}$ by:
\begin{align*}
\mbf{U}_{Q,\leq\mbf{m}}^{\pm}, \mbf{U}_{Q,\geq\mbf{m}}^{\pm}\subset\mbf{U}^{\pm}_{Q}
\end{align*}
and the graded pieces:
\begin{align*}
\mbf{U}_{Q,\leq\mbf{m}|\pm\mbf{n},\pm d}^{\pm}=\mbf{U}_{Q,\pm\mbf{n},\pm d}^{\pm}\cap\mbf{U}_{Q,\leq\mbf{m}}^{\pm},\qquad\mbf{U}_{Q,\geq\mbf{m}|\pm\mbf{n},\pm d}^{\pm}=\mbf{U}_{Q,\pm\mbf{n},\pm d}^{\pm}\cap\mbf{U}_{Q,\geq\mbf{m}}^{\pm}.
\end{align*}

By definition it is easy to see that $\mbf{U}_{\leq\mbf{m}|\pm\mbf{n},\pm d}$ are finite-dimensional for any $(\mbf{n},d)\in\mbb{N}^I\times\mbb{Z}$.

Also we denote the subspace of shuffle element of slope between $\mbf{m}_1$ and $\mbf{m}_2$ as:
\begin{align*}
\mbf{U}^{\pm}_{Q,[\mbf{m}_1,\mbf{m}_2]}:=\mbf{U}^{\pm}_{Q,\geq\mbf{m}_1}\cap\mbf{U}^{\pm}_{Q,\leq\mbf{m}_2}
\end{align*}
and its graded pieces:
\begin{align*}
\mbf{U}^{\pm}_{Q,[\mbf{m}_1,\mbf{m}_2]|\pm\mbf{n},\pm d}=\mbf{U}^{\pm}_{Q,\pm\mbf{n},\pm d}\cap\mbf{U}^{\pm}_{Q,[\mbf{m}_1,\mbf{m}_2]}.
\end{align*}

The following lemma is easy to prove:

\begin{lem}
$\mbf{U}_{Q,\leq\mbf{m}}^{\pm}$, $\mbf{U}_{Q,\geq\mbf{m}}^{\pm}$ and $\mbf{U}^{\pm}_{Q,[\mbf{m}_1,\mbf{m}_2]}$ are subalgebras of $\mbf{U}_{Q}^{\pm}$.
\end{lem}

We will say that $F\in\mbf{U}^+_{Q}$ and $G\in\mbf{U}^-_{Q}$ have naive slope $\mbf{m}$ if:
\begin{align*}
&\text{vdeg }F\leq\mbf{m}\cdot\text{hdeg }F\\
&\text{vdeg }G\geq\mbf{m}\cdot\text{hdeg }G.
\end{align*}

The following proposition has been proved in \cite{N22}:

\begin{prop}
We have the following result for the action of the topological coproduct:
\begin{align*}
\Delta(\mbf{U}^{+}_{Q,\leq\mbf{m}})\subset\mbf{U}_{Q}^0\mbf{U}^+_{Q,\leq\mbf{m}}\otimes^{V}\mbf{U}_{Q}^+,\qquad\Delta(\mbf{U}^{-}_{Q,\leq\mbf{m}})\subset\mbf{U}^-_{Q}\otimes^{V}\mbf{U}_{Q,\leq\mbf{m}}^-\mbf{U}_{Q}^0
\end{align*}
where $\otimes^V$ is defined in \eqref{V-and-H-coproduct}.
\end{prop}

\begin{defn}
For any $\mbf{m}\in\mbb{R}^{I}$, we define the \textbf{slope subalgebra} $\mc{B}_{\mbf{m}}^{\pm}$ as the subalgebra of $\mbf{U}_{Q}^{\pm}$ written as:
\begin{align*}
\mc{B}_{\mbf{m}}^{\pm}=\bigoplus_{\mbf{n}\in\mbb{N}^I}\mc{B}_{\mbf{m}|\pm\mbf{n}}^{\pm}=\bigoplus_{\mbf{n}\in\mbb{N}^I}^{\mbf{m}\cdot\mbf{n}\in\mbb{Z}}\mbf{U}_{Q,\leq\mbf{m}|\pm\mbf{n},\pm\mbf{m}\cdot\mbf{n}}^{\pm}.
\end{align*}
\end{defn}

Now we define the vector spaces:
\begin{align*}
\mc{B}_{\mbf{m}}^{\geq}=\mc{B}_{\mbf{m}}^{+}\otimes\mbb{K}[h_{i,0}^{+}]_{i\in I}\\
\mc{B}_{\mbf{m}}^{\leq}=\mc{B}_{\mbf{m}}^{-}\otimes\mbb{K}[h_{i,0}^{-}]_{i\in I}.
\end{align*}
with the following commutating relations:
\begin{align*}
h_{i,0}^+X=Xh_{i,0}^+\gamma_{\mbf{e}_i,\text{hdeg }X},\qquad h_{i,0}^-X=Xh_{i,0}^{-}\gamma_{\text{hdeg }X,\mbf{e}_i}^{-1}
\end{align*}
and here $\gamma_{\mbf{m},\mbf{n}}$ is defined in \eqref{important-constant}.

The algebra $\mc{B}_{\mbf{m}}^{\geq,\leq}$ also have the coproduct structure $\Delta_{\mbf{m}}$, and it is defined as follows: For the Cartan elements:
\begin{align*}
\Delta_{\mbf{m}}(h_{i,0}^{\pm})=h_{i,0}^{\pm}\otimes h_{i,0}^{\pm}
\end{align*}
and for any $F\in\mc{B}_{\mbf{m}|\mbf{n}}$ and $G\in\mc{B}_{\mbf{m}|-\mbf{n}}$:
\begin{equation}\label{slope-m-coproduct-slope-subalgebra}
\begin{aligned}
&\Delta_{\mbf{m}}(F)=\sum_{\mbf{0}\leq\mbf{k}\leq\mbf{n}}\lim_{\xi\rightarrow\infty}\frac{h_{\mbf{n}-\mbf{k}}F(\cdots,z_{i1},\cdots,z_{ik_i}\otimes\xi z_{i,k_{i}+1},\cdots,\xi z_{in_i},\cdots)}{\xi^{\mbf{m}\cdot(\mbf{n}-\mbf{k})}\text{lead}[\prod_{1\leq a\leq k_i}^{i\in I}\prod_{k_j<b\leq n_j}^{j\in I}\zeta_{ji}(\xi z_{jb}/z_{ia})]}\\
&\Delta_{\mbf{m}}(G)=\sum_{\mbf{0}\leq\mbf{k}\leq\mbf{n}}\lim_{\xi\rightarrow\infty}\frac{F(\cdots,\xi z_{i1},\cdots,\xi z_{ik_i}\otimes z_{i,k_{i}+1},\cdots, z_{in_i},\cdots)h_{-\mbf{k}}}{\xi^{-\mbf{m}\cdot\mbf{k}}\text{lead}[\prod_{1\leq a\leq k_i}^{i\in I}\prod_{k_j<b\leq n_j}^{j\in I}\zeta_{ji}(\xi z_{ia}/z_{jb})]}.
\end{aligned}
\end{equation}

Here we use the notation:
\begin{align*}
h_{\pm\mbf{n}}=\prod_{i\in I}h_{i,0}^{\pm n_i}.
\end{align*}

Now with respect to the bialgebra structure above the restriction of the nondegenerate pairing in \eqref{nondegenerate-bialgebra-pairing}
\begin{align}\label{slope-subalgebra-pairing}
\mc{B}_{\mbf{m}}^{\geq}\otimes\mc{B}_{\mbf{m}}^{\leq}\rightarrow\mbb{K}
\end{align}
the bialgebra pairing implies that we can define the Drinfeld double:
\begin{align*}
\mc{B}_{\mbf{m}}:=\mc{B}_{\mbf{m}}^{\geq}\otimes\mc{B}_{\mbf{m}}^{\leq}
\end{align*}
and this is a Hopf algebra with a natural universal $R$-matrix $q^{\Omega}R_{\mbf{m}}^+\in\mc{B}_{\mbf{m}}^{\geq}\hat{\otimes}\mc{B}_{\mbf{m}}^{\leq}$ induced by the bialgebra pairing \eqref{slope-subalgebra-pairing}. Here $q^{\Omega}$ is defined in \eqref{Cartan-R-matrix}.

We define $q^{\Omega}R_{\mbf{m}}^-:=(q^{\Omega}R_{\mbf{m}}^+)_{21}\in\mc{B}_{\mbf{m}}^{\leq}\hat{\otimes}\mc{B}_{\mbf{m}}^{\geq}$ to be the opposite of the universal $R$-matrix $R_{\mbf{m}}^+$. It is easy to check that $(q^{\Omega}R_{\mbf{m}}^-)^{-1}\in\mc{B}_{\mbf{m}}^{\leq}\hat{\otimes}\mc{B}_{\mbf{m}}^{\geq}$ is also the universal $R$-matrix with respect to the coproduct $\Delta_{\mbf{m}}$.

\begin{prop}
The inclusion map $\mc{B}_{\mbf{m}}\subset\mbf{U}_{Q}$ is an algebra monomorphism
\end{prop}
\begin{proof}
The proof is similar to Proposition 3.10 in \cite{N22}. We consider any $a\in\mc{B}_{\mbf{m}}^{\geq}$ and $b\in\mc{B}_{\mbf{m}}^{\leq}$. Let us write
\begin{align*}
\Delta(a)=\sum_{s\in S}a_{1,s}\otimes a_{2,s},\qquad\Delta(b)=\sum_{t\in T}b_{1,t}\otimes b_{2,t}
\end{align*}
for some sets $S$ and $T$. Let us also write:
\begin{align*}
\Delta_{\mbf{m}}(a)=\sum_{s\in S'}a_{1,s}\otimes a_{2,s},\qquad\Delta_{\mbf{m}}(b)=\sum_{t\in T'}b_{1,t}\otimes b_{2,t}
\end{align*}
where $S'\subset S$ and $T'\subset T$ such that for $s\in S'$ and $t\in T'$ such that:
\begin{equation*}
\begin{aligned}
&\text{vdeg }a_{1,s}=\mbf{m}\cdot(\text{hdeg }a_{1,s})\Leftrightarrow\text{vdeg }a_{2,s}=\mbf{m}\cdot(\text{hdeg }a_{2,s})\\
&\text{vdeg }b_{1,t}=\mbf{m}\cdot(\text{hdeg }b_{1,t})\Leftrightarrow\text{vdeg }b_{2,t}=\mbf{m}\cdot(\text{hdeg }b_{2,t})
\end{aligned}
\end{equation*}

We also have the following relations:
\begin{align*}
\sum_{s\in S,t\in T}a_{1,s}b_{1,t}\langle a_{2,s},b_{2,t}\rangle=\sum_{s\in S,t\in T}b_{2,t}a_{2,s}\langle a_{1,s},b_{1,t}\rangle
\end{align*}
and since $a_{2,s}$ and $b_{1,t}$ have naive slope $\leq\mbf{m}$, for all $s\in S$ and $t\in T$. This implies that:
\begin{equation}\label{naive-slope-degree-comparison}
\begin{aligned}
&\text{vdeg }a_{2,s}\leq\mbf{m}\cdot(\text{hdeg }a_{2,s})\Rightarrow\text{vdeg }a_{1,s}\geq\mbf{m}\cdot(\text{hdeg }a_{1,s})\\
&\text{vdeg }b_{1,t}\geq\mbf{m}\cdot(\text{hdeg }b_{1,t})\Rightarrow\text{vdeg }b_{2,t}\leq\mbf{m}\cdot(\text{hdeg }b_{2,t})
\end{aligned}
\end{equation}
We know that the pairing is nonzero only if they have the equality on \ref{naive-slope-degree-comparison}. Therefore \ref{naive-slope-degree-comparison} can be reduced to:
\begin{align*}
\sum_{s\in S',t\in T'}a_{1,s}b_{1,t}\langle a_{2,s},b_{2,t}\rangle=\sum_{s\in S',t\in T'}b_{2,t}a_{2,s}\langle a_{1,s},b_{1,t}\rangle.
\end{align*}
Thus the proof is finished.
\end{proof}

The shift automorphism restricts to the following isomorphism:
\begin{equation*}
\begin{tikzcd} 
\tau_{\mbf{k}}:\mc{B}_{\mbf{m}}\arrow[r,"\cong"]&\mc{B}_{\mbf{m}+\mbf{k}}.
\end{tikzcd}
\end{equation*}

We also define the slope $\infty$ subalgebra by the following:
\begin{align*}
\mc{B}_{\infty}^{\geq}=\mbb{K}[h_{i,d}^{+}]_{i\in I,d\geq0},\qquad\mc{B}_{\infty}^{\leq}=\mbb{K}[h_{i,d}^{-}]_{i\in I,d\leq0}
\end{align*}
and the coproduct $\Delta_{\infty}$ is defined as the restriction of $\Delta$ to $\mc{B}^{\geq}_{\infty}$ and $\mc{B}^{\leq}_{\infty}$.

In this paper we also assume that the bialgebra pairing
\begin{align*}
\langle-,-\rangle:\mc{B}_{\infty}^{\geq}\otimes\mc{B}_{\infty}^{\leq}\rightarrow\mbb{K}
\end{align*}
is nondegenerate. This condition is generally true for the case that we are concerning about.

\subsubsection{Slope factorizations}
One of the main property for the quantum loop group is the slope factorization. 

We denote a parametrised curve $\bm{\mu}:\mbb{R}\rightarrow\mbb{R}^I$ as the catty-corner curve (See \cite{N26}) if 
\begin{align*}
t_1<t_2\text{ implies }\mu_i(t_1)<\mu_i(t_2)
\end{align*}
and
\begin{align*}
\lim_{t\rightarrow\pm\infty}\mu_i(t)=\pm\infty
\end{align*}

For the proof of the following theorem, one can refer to \cite{N22}\cite{N26}:
\begin{thm}
For arbitrary catty-corner curve $\bm{\mu}:\mbb{R}\rightarrow\mbb{R}^I$ in $\mbb{R}^I$, there is an isomorphism:
\begin{align*}
\mbf{U}_{Q}^{\pm}\cong\bigotimes^{\rightarrow}_{t\in\mbb{R}}\mc{B}^{\pm}_{\bm{\mu}(t)}
\end{align*}
given by the multiplication, where $\rightarrow$ means that we take the tensor product in increasing order of $t$. This isomorphism preserves the pairing \eqref{nondegenerate-bialgebra-pairing} in the sense that:
\begin{align*}
\langle\prod_{t\in\mbb{R}}^{\rightarrow}E_t,\prod_{t\in\mbb{R}}^{\rightarrow}F_t\rangle=\prod_{t\in\mbb{R}}\langle E_t,F_t\rangle
\end{align*}
\end{thm}

Similarly we also we have that:
\begin{align*}
\mbf{U}_{Q,\leq\bm{\mu}(t_1)}^{\pm}\cong\bigotimes^{\rightarrow}_{t\leq t_1}\mc{B}_{\bm{\mu}(t)}^{\pm},\qquad\mbf{U}_{Q,\geq\bm{\mu}(t_1)}^{\pm}\cong\bigotimes^{\rightarrow}_{t\geq t_1}\mc{B}_{\bm{\mu}(t)}^{\pm},\qquad\mbf{U}_{Q,[\bm{\mu}(t_1),\bm{\mu}(t_2)]}^{\pm}\cong\bigotimes^{\rightarrow}_{t\in[t_1,t_2]}\mc{B}_{\bm{\mu}(t)}^{\pm}
\end{align*}

Moreover, the restriction of the nondegenerate bialgebra pairing \eqref{nondegenerate-bialgebra-pairing} will give the nondegenerate bialgebra pairing:
\begin{equation*}
\begin{aligned}
&\mbf{U}_{Q,\leq\bm{\mu}(t_1)}^{+}\otimes\mbf{U}_{Q,\leq\bm{\mu}(t_1)}^{-}\rightarrow\mbb{K},\qquad\mbf{U}_{Q,\geq\bm{\mu}(t_1)}^{+}\otimes\mbf{U}_{Q,\geq\bm{\mu}(t_1)}^{-}\rightarrow\mbb{K}\\
&\mbf{U}_{Q,[\bm{\mu}(t_1),\bm{\mu}(t_2)]}^+\otimes\mbf{U}_{Q,[\bm{\mu}(t_1),\bm{\mu}(t_2)]}^-\rightarrow\mbb{K}
\end{aligned}
\end{equation*}

These nondegenerate bialgebra pairing admits the following universal $R$-matrix:
\begin{align}\label{factorization-of-R-matrices}
R^{\leq\bm{\mu}(t_1)}=\prod_{t\leq t_1}^{\rightarrow}R^{+}_{\bm{\mu}(t)},\qquad R^{\geq\bm{\mu}(t_1)}=\prod_{t\geq t_1}^{\rightarrow}R^{+}_{\bm{\mu}(t)},\qquad R^{[\bm{\mu(t_1)},\bm{\mu}(t_2)]}=\prod_{t_1\leq t\leq t_2}R^+_{\bm{\mu}(t)}
\end{align}
which is independent of the catty-corner curve $\bm{\mu}(t)$.

\subsection{Drinfeld double for $\Delta_{(\mbf{m})}$}
Following the logic of \cite{N26}. We consider the following subalgebra:
\begin{align}\label{slope-m-paired-subalgebras}
\mc{A}^{\geq\mbf{m}}=\mbf{U}_{Q,\geq\mbf{m}}^+\otimes\mbf{U}_{Q,<\mbf{m}}^-,\qquad\mc{A}^{\leq\mbf{m}}=\mbf{U}_{Q,\geq\mbf{m}}^-\otimes\mbf{U}_{Q,<\mbf{m}}^+
\end{align}

In \cite{N26}, Negu\c{t} introduced a new new coproduct $\Delta_{(\mbf{m})}$ corresponding to the pair of subalgebras. 

\begin{thm}[See Theorem 3.11 in \cite{N26}]\label{new-new-coproduct:theorem}
There is a coproduct $\Delta_{(\mbf{m})}$ on $\mbf{U}_{Q}$ which is the coproduct with respect to the pairing $\langle-,-\rangle:\mc{A}^{\geq\mbf{m}}\otimes\mc{A}^{\leq\mbf{m}}\rightarrow\mbb{K}$. Moreover we have that:
\begin{align*}
\Delta_{(\mbf{m})}:\mc{A}^{\geq\mbf{m}}\rightarrow\mc{A}^{\geq\mbf{m}}\otimes^H\mc{A}^{\geq\mbf{m}},\qquad\Delta_{(\mbf{m})}:\mc{A}^{\leq\mbf{m}}\rightarrow\mc{A}^{\leq\mbf{m}}\otimes^H\mc{A}^{\leq\mbf{m}}
\end{align*}
\end{thm}
Moreover, this new new coproduct $\Delta_{(\mbf{m})}$ is preserved when restricted to $\mc{B}_{\mbf{m}}$:

\begin{lem}
When restricted to the slope subalgebra $\mc{B}_{\mbf{m}}$, the coproduct $\Delta_{(\mbf{m})}$ coincides with the coproduct $\Delta_{\mbf{m}}$ defined in \eqref{slope-m-coproduct-slope-subalgebra}.
\end{lem}
\begin{proof}
We will just refer to two methods proving this. One can mimick the proof given in Theorem 3.11 of \cite{N26} when restricted to $\mc{B}_{\mbf{m}}$ to show that the coproduct $\Delta_{(\mbf{m})}$ is preserved. Another way is to use the formula given in \eqref{Coproduct-of-between-slope-one-and-Drinfeld-one}, and use the method in Proposition 2.3 in \cite{Z23} to prove the lemma.
\end{proof}

\subsection{Universal $R$-matrix and factorizations}
\subsubsection{Factorization for the universal $R$-matrix}
From the above persepctive, we can see that the coproduct $\Delta_{(\mbf{m})}$ comes from the bialgebra pairing:
\begin{align*}
\langle-,-\rangle:\mc{A}^{\geq\mbf{m}}\otimes\mc{A}^{\leq\mbf{m}}\rightarrow\mbb{K}
\end{align*}
which is induced from the bialgebra pairing $\langle-,-\rangle:\mbf{U}_{Q}^{\geq}\otimes\mbf{U}_{Q}^{\leq}\rightarrow\mbb{K}$. In this case one can induce a universal $R$-matrix $\mc{R}_{\mbf{m}}$ with respect to the above bialgebra pairing. By the factorisation given above, we have that:
\begin{align*}
\mc{R}^{\mbf{m}}=R^{\geq\mbf{m}}R^{\infty}R^{-,<\mbf{m}}\in\mc{A}^{\geq\mbf{m}}\hat{\otimes}\mc{A}^{\leq\mbf{m}}
\end{align*}
with $R^{\geq\mbf{m}}$, $R^{-,<\mbf{m}}:=(R^{<\mbf{m}})_{12}$ defined in \eqref{factorization-of-R-matrices} and $R^{\infty}$ defined in \eqref{Infinity-R-matrix}.

Similar argument as above shows that 
$$(\mc{R}^{\mbf{m}})_{21}^{-1}=(R^{<\mbf{m}})^{-1}(R^{\infty})^{-1}(R^{-,>\mbf{m}})^{-1}\in\mc{A}^{\leq\mbf{m}}\hat{\otimes}\mc{A}^{\geq\mbf{m}}$$
is also the universal $R$-matrix with respect to the coproduct structure $\Delta_{(\mbf{m})}$.

Given $\mbf{m}'>\mbf{m}$, we have the following relation for the universal $R$-matrix:
\begin{align}\label{universal-R-matrix-of-different-slopes}
\mc{R}^{\mbf{m}'}=(R^{[\mbf{m},\mbf{m}')})^{-1}\mc{R}^{\mbf{m}}R^{-,[\mbf{m},\mbf{m}')},\qquad(\mc{R}^{\mbf{m}'})_{21}^{-1}=(R^{[\mbf{m},\mbf{m}')})^{-1}(\mc{R}^{\mbf{m}})_{21}^{-1}(R^{-,[\mbf{m},\mbf{m}')})
\end{align}
and the relation between the coproduct $\Delta_{(\mbf{m})}$ and $\Delta_{(\mbf{m}')}$ can be given by:
\begin{align}\label{Relation-between-coproducts-of-different-slope}
\Delta_{(\mbf{m}')}(F)=(R^{-,[\mbf{m},\mbf{m}')})^{-1}\Delta_{(\mbf{m})}(F)R^{-,[\mbf{m},\mbf{m}')}.
\end{align}

Furthermore, the relation between the coproduct $\Delta_{(\mbf{m})}$ and the Drinfeld coproduct $\Delta$ defined in \eqref{Drinfeld-coproduct-definition} is given by:
\begin{align}\label{Coproduct-of-between-slope-one-and-Drinfeld-one}
\Delta_{(\mbf{m})}(F)=(R^{-,<\mbf{m}})^{-1}\Delta(F)R^{-,<\mbf{m}}.
\end{align}

\subsubsection{Infinite slope $R$-matrix}
The subtlety above is to construct the universal $R$-matrix for $\mc{B}_{\infty}$. One can factorise the universal $R$-matrix $R^{\infty}$ as:
\begin{align}\label{Infinity-R-matrix}
R^{\infty}=R_{\infty}'R_{\infty}''
\end{align}
and here $R_{\infty}''$ is the part generated by $h_{i,d}^{\pm}$ with $d\geq1$, it is relatively easy to see that the bialgebra pairing is nondegenerate. However, if we restrict to the subalgebra generated by $h_{i,0}^{\pm}$, the bialgebra pairing restricted to this Cartan subalgebra is degenerate, and this should correspond to $R_{\infty}'$.

To deal with the subtleties, it is convenient to set a new formal parametre $\hbar$ such that $h_{i,0}^{\pm}=e^{\pm\hbar H_i}$ and $\gamma_{ij}:=e^{\hbar\pi_{ij}}$. In this case we will have the bialgebra pairing:
\begin{align}\label{renewed-cartan-bialgebra-pairing}
\mbb{K}[[\hbar]][H_i]_{i\in I}\otimes\mbb{K}[[\hbar]][H_i]_{i\in I}\rightarrow\mbb{K}[[\hbar]],\qquad\langle H_i,H_j\rangle=\hbar\pi_{ij}.
\end{align}

Throughout the paper we will assume that the matrix $(\pi_{ij})_{i,j\in I}$ is nondegenerate.  Thus one can write down the universal $R$-matrix $R_{\infty}'$ as:
\begin{align*}
R_{\infty}'=\exp(\hbar\sum_{i,j\in I}(\pi^{-1})_{ij}H_i\otimes H_j).
\end{align*}

In our paper we denote 
\begin{align*}
q:=e^{\hbar}
\end{align*}
and we also denote
\begin{align}\label{Cartan-R-matrix}
q^{\Omega}:=R_{\infty}',\qquad\Omega=\sum_{i,j\in I}(\pi^{-1})_{ij}H_i\otimes H_j
\end{align}
which was used in \cite{OS22}.

\subsection{Equivariant parametres and automorphisms}\label{sub:equivariant_parametres_and_automorphisms}
Fix a generic parametre $a\in\mbb{C}^*$, we define the automorphism $D_a$ of $\mbf{U}_{Q}\otimes\mbb{K}[[a^{\pm1}]]$ via:
\begin{equation*}
\begin{aligned}
&D_a(F(z))=a^{\pm d}F(z),\qquad F(z)\in\mbf{U}_{Q,\pm d}^{\pm}\\
&D_a(h_{i,\pm d}^{\pm})=a^{\pm d}h_{i,\pm d}^{\pm}
\end{aligned}
\end{equation*}
and we will call the parametre $a$ as the \textbf{equivariant parametre}.

In this way one can define the corresponding universal $R$-matrix with spectral parameter by:
\begin{align*}
\mc{R}^{s}(a):=(D_a\otimes\text{Id})\mc{R}^s\in\mc{S}\hat{\otimes}\mc{S}[[a]]
\end{align*}
and
\begin{align*}
R_{\mbf{m}}^{\pm}(a):=(D_a\otimes\text{Id})R_{\mbf{m}}^{\pm}\in\mc{B}_{\mbf{m}}\hat{\otimes}\mc{B}_{\mbf{m}}[[a^{\pm1}]]
\end{align*}

\section{\textbf{Category $\mc{O}$ for the quantum loop group}}\label{sec:_textbf_category_mc_o}

In this section we introduce the category $\mc{O}$ for the quantum loop group $\mbf{U}_{Q}$. This was first introduced in \cite{MY14} for the case of the quantum affine algebras. 

\begin{defn}
An $\ell$-weight is an $I$-tuple of sequences of complex numbers
\begin{align*}
\bm{\psi}:=(\psi^{\pm}_{i,\pm d})_{i\in I,d\geq0}
\end{align*}
such that $\psi_{i,0}^+\psi_{i,0}^-=1$ for every $i\in I$.
\end{defn}

\begin{defn}\label{classification-of-weight:category_mc_o}
An $\ell$-weight is a collection of invertible power series
\begin{align*}
\bm{\psi}=(\psi_{i}^{\pm}(z)=\sum_{d=0}^{\infty}\frac{\psi_{i,\pm d}^{\pm}}{z^{\pm d}}\in\mbb{C}[[z^{\mp1}]]^*)
\end{align*}
such a $\bm{\psi}$ is called
\begin{itemize}
	\item \textbf{rational} if every $\psi_i^{\pm}(z)$ is the expansion of a rational function.
	\item \textbf{regular} if every $\psi_i^{\pm}(z)$ is the expansion of a rational function which is regular at $z=0$.
	\item \textbf{polynomial} if every $\psi_i^{\pm}(z)$ is a polynomial in $z^{-1}$.
\end{itemize}
\end{defn}

\subsection{Category $\mc{O}$}
We consider the $\mbf{U}_{Q}$-module $V$ such that the Cartan subalgebra $\{h_{i,0}^{\pm}\}_{i\in I}$ acts diagonalisably:
\begin{align*}
V=\bigoplus_{\bm{\lambda}=(\lambda_i)_{i\in I}(\mbb{K}^*)^I}V_{\bm{\lambda}}
\end{align*}
where:
\begin{align*}
V_{\bm{\lambda}}=\{v\in V|h_{i,0}^+\cdot v=\lambda_i v\}.
\end{align*}

We say $V\in\mc{O}$ is a highest $\ell$-weight representation of highest weight $\bm{\psi}$ if $V=\mbf{U}_{Q}\cdot v$ for some singular $\ell$-weight $v\in V$. We call $v$ the highest $\ell$-weight vector $V$.

By the following commutation relations:
\begin{align*}
h_{i,0}^+X=Xh_{i,0}^+\gamma_{\mbf{e}^i,\text{hdeg }X}
\end{align*}
we have, for any $X\in\mbf{U}_{Q}$,
\begin{align*}
X:V_{\bm{\lambda}}\rightarrow V_{\bm{\lambda}\bm{\gamma}^{\text{hdeg }X}}
\end{align*}
where $\bm{\gamma}^{\mbf{n}}$ denotes the weight with $i$-th component $\gamma_{\mbf{e}^i,\mbf{n}}$ for all $\mbf{n}\in\mbb{Z}^I$. 

\begin{defn}
A $\mbf{U}_{Q}$-module $V$ is said to be in category $\mc{O}$ if it has a weight decomposition above with every $V_{\bm{\lambda}}$ finite-dimensional and nonzero only for
\begin{align*}
\bm{\lambda}\in\{\bm{\lambda}^1\bm{\gamma}^{\mbf{n}},\cdots,\bm{\lambda}^k\bm{\gamma}^{\mbf{n}}\}_{\mbf{n}\in\mbb{N}^I}
\end{align*}
for finitely many weights $\bm{\lambda}^1,\cdots,\bm{\lambda}^k$.
\end{defn}

\textbf{Remark.} The category $\mc{O}$ corresponds to the affine category $\hat{\mc{O}}$ defined in \cite{MY14}. In this paper we only use the notation $\mc{O}$ as the affine version mentioned in that paper, and we will not consider the category $\mc{O}$ for the slope subalgebras.

\subsubsection{Simples modules by quotients}
We say that $V\in\mc{O}$ is a highest $\ell$-weight representation of highest $\ell$-weight $\bm{\psi}$ if $V=\mbf{U}_{Q}\cdot v$ for some vector $v\in V_{\bm{\lambda}}$ such that $h_{i}^{\pm}(z)\cdot v=\psi_{i}^{\pm}(z)\cdot v$.

For every rational $\ell$-weight $\bm{\psi}$, we denote by $L(\bm{\psi})$ the simple $\mbf{U}_{Q}$-module with highest $\ell$-weight $\bm{\psi}$. Obviously that $L(\bm{\psi})$ and $L(\bm{\psi}')$ are isomorphic if $\bm{\psi}=\bm{\psi}'$.

We denote $\mc{R}$ the set of rational $\ell$-weights as in Definition \ref{classification-of-weight:category_mc_o}. 
\begin{thm}
The map $\bm{\psi}\mapsto L(\bm{\psi})$ defines a bijection between $\mc{R}$ and the isomorphism classes of simple objects in $\mc{O}$.
\end{thm}
\begin{proof}
The proof is similar to the proof of Theorem 3.6 in \cite{MY14}. Here we make a review of the proof.

Suppose that $V\in\mc{O}$ is irreducible, then $V$ contains a singular $\ell$-weight vetcor $v$. Since $V$ is irreducible, $V=\mbf{U}_{Q}\cdot v$ and so $V$ is a highest $\ell$-weight representation. 

It is enough to show that a highest $\ell$-weight irreducible representation $V$ is in $\mc{O}$ if and only if its highest $\ell$-weight $\bm{\psi}$ is rational.

Now we first suppose that $\text{dim }(V_{\mbf{e}_i})<\infty$. This implies that there exists $N>0$ and $a_1,\cdots,a_N\in\mbb{K}$ such that $\sum_{j=1}^{N}a_{j}e_{i,j}\cdot v=0$. Thus we have that for all $s\in\mbb{Z}$:
\begin{align*}
0=f_{i,s}\sum_{j=1}^Na_{j}e_{i,j}\cdot v=\sum_{j=1}^Na_j(h_{i,j+s}^{+}-h_{i,-j-s}^{-})
\end{align*}
which means that $\sum_{j=1}^Na_j\psi_{i,j+s}^+=0$ for all $s\geq0$. Also note that:
\begin{equation*}
\begin{aligned}
\psi_i^+(z)\sum_{j=1}^Na_jz^{N-j}=\sum_{j=1}^N\sum_{s=-j}^{\infty}z^{s+N}a_j\psi_{i,j+s}^+=\sum_{j=1}^N\sum_{s=-j}^{-1}z^{s+N}a_{j}\psi_{i,j+s}^+=
\sum_{j=1}^Nb_jz^{N-j},\qquad b_j=\sum_{\ell=1}^Na_{\ell}\psi_{i,\ell-j}^+
\end{aligned}
\end{equation*}
and therefore we have that:
\begin{align*}
\psi_{i}^+(z)=\frac{\sum_{j=1}^Nb_jz^{N-j}}{\sum_{j=1}^Na_jz^{N-j}}.
\end{align*}

Similarly $0=\sum_{j=1}^Na_j\psi_{i,j+s}^-$ for all $s<-N$. Letting $\psi_{i}^-(z):=\sum_{n=0}^{\infty}z^{-n}\psi_{i,-n}^-$, we have
\begin{align*}
\psi_{i}^-(z)\sum_{j=1}^Na_jz^{N-j}=\sum_{j=1}^Nb_j'z^{N-j},\qquad b_j'=\sum_{\ell=1}^Na_{\ell}\psi_{i,\ell-j}^-
\end{align*}
and thus we have that:
\begin{align*}
\psi_{i}^-(z)=\frac{\sum_{j=1}^Nb_j'z^{N-j}}{\sum_{j=1}^Na_jz^{N-j}}
\end{align*}
and this implies that $\psi_{i}^+(z)$ and $\psi_{i}^-(z)$ are rational functions.

Conversely, suppose that $\psi_{i}^+(z)$ and $\psi_i^-(z)$ are as above for some $N>0$. Then a similar calculation shows that $f_{i,s}\sum_{j=1}^Na_je_{i,M+j}\cdot v=0$ for all $M\in\mbb{Z}$. Since $V$ has no singular vectors which are not scalar multiples of $v$, it follows that $\sum_{j=1}^Na_{j}e_{i,M+j}\cdot v=0$. By applying this result finitely many times, any given vector $e_{i,r}v$ can be expressed as a linear combination of the vectors $\{e_{i,j}v|j=1,\cdots,N\}$, which means that $V_{\bm{\lambda}\bm{\gamma}^{-\mbf{e}_i}}$ is finite-dimensional.

Now if $V_{\mbf{e}_i}$ is finite-dimensional for every $i\in I$. It remains to prove that the weighted space $V_{\bm{\alpha}}$ is finite-dimensional too. Now we suppose that the statement is true for $\bm{\beta}<\bm{\alpha}$. Thus without loss of assumption, we assume that the statement is true for $V_{\bm{\beta}}$, and for $V_{\bm{\beta}+\mbf{e}_i}$. The induction follows from the proof of case (b) in \cite{CP94}.

For the only if part, one can refer to the Lemma 4.11 in \cite{Her04}.

\end{proof}

\subsection{Line bundle operator for semisimple modules}
Recall that we have defined an automorphism map $\tau_{\mbf{k}}:\mbf{U}_{Q}\rightarrow\mbf{U}_{Q}$ for the quantum loop group. For simplicity, we now assume that $V$ is a simple $\mbf{U}_{Q}$-module. By the automorphism map $\tau_{\mbf{k}}:\mbf{U}_{Q}\rightarrow\mbf{U}_{Q}$ constructed in \eqref{shift-automorphism-shuffle-big}, \eqref{negative-shift-automorphism-shuffle-big} and Lemma \ref{shifted-automorphism-on-quantum-loop-group:lemma}, the automorphism $\tau_{\mbf{k}}$ will induce a new $\mbf{U}_{Q}$-module structure on $V$, and we denote it by $V^{\tau_{\mbf{k}}}$.

\begin{lem}\label{line-bundle-operator:lemma}
Suppose that $V\in\mc{O}$ is a simple module or the tensor product or the direct sum of simple modules in $\mc{O}$. There is an isomorphism of $\mbf{U}_{Q}$-modules $V\cong V^{\tau_{\mbf{k}}}$ which intertwines the automorphism $\tau_{\mbf{k}}$ in Lemma \ref{shifted-automorphism-on-quantum-loop-group:lemma} for arbitrary $\mbf{k}\in\mbb{Z}^I$ , and we denote the automorphism by $\mc{L}_{\mbf{k}}$
\end{lem}
\begin{proof}
Let us first prove the lemma for $V=L(\bm{\psi})$. It is easy to see that $L(\bm{\psi})^{\tau_{\mbf{k}}}$ is also a simple module with the highest $\ell$-weight given by $\bm{\psi}$. By the above theorem we know that there is an isomorphism of $\mbf{U}$-modules $L(\bm{\psi})\cong L(\bm{\psi})^{\tau_{\mbf{k}}}$. By the Schur lemma, we have an invertible morphism and we denote it by $\mc{L}_{\mbf{k}}$.

The case for the direct sum of simple modules is just taking the diagonal sum of $\sum\mc{L}_{\mbf{k}}$. For the case of the tensor product of simple modules we just take the tensor product $\otimes\mc{L}_{\mbf{k}}$ of $\mc{L}_{\mbf{k}}$.

\end{proof}

\textbf{Remark. }The terminology of the line bundle operator comes from \cite{OS22} that such an intertwiner operator corresponds to the line bundle multiplication in the $K$-theory of the Nakajima quiver varieties.

\section{\textbf{Dynamical Weyl groups for shuffle algebras}}\label{sec:_textbf_dynamical_weyl_groups_for_shuffle_algebras}
\subsection{Fusion operators}
Given a simple module $V\in\mc{O}$ in the category $\mc{O}$. Since $V$ is the highest loop $\ell$-weight module, we can simply write $V=\bigoplus_{\mbf{n}\in\mbb{N}^I}V_{\mbf{n}}$ with $V_{\mbf{n}}:=V_{\bm{\lambda}\bm{\gamma}^{-\mbf{n}}}$. Given a set of parametres $(z_1,\cdots,z_{I})\in(\mbb{C}^*)^{I}$, we define the operator $z$ as:
\begin{align*}
z\cdot v=\prod_{i\in I}z_{i}^{n_i}v,\qquad v\in V_{\mbf{n}}
\end{align*}
with $\mbf{n}=(n_1,\cdots,n_{I})$. We also write
\begin{align*}
q^{\Omega}=e^{\sum_{i\in I}(\pi^{-1})_{ij}H_{i,0}^+\otimes H_{j,0}^{-}}
\end{align*}
as defined in \eqref{Cartan-R-matrix}.

We refer to the parametres $z=(z_1,\cdots,z_{I})$ as the \textbf{Kahler parametres}, following the terminology of \cite{O15} in the $K$-theoretic quasimap counting theory for the Nakajima quiver varieties.

For the weight modules $V=\bigoplus_{\mbf{n}\in\mbb{N}^I}V_{\bm{\lambda}\gamma^{-\mbf{n}}}$ in the category $\mc{O}$, we will often suppress their corresponding highest weight $\bm{\lambda}$ and simply write $V=\bigoplus_{\mbf{n}\in\mbb{N}^I}V_{\mbf{n}}$ when no confusion can arise.

Given two weight modules $V=\bigoplus_{\mbf{n}\in\mbb{N}^I}V_{\mbf{n}}$ and $W=\bigoplus_{\mbf{n}\in\mbb{N}^I}W_{\mbf{n}}$ in $\mc{O}$, and let $F\in\text{End}(V\otimes W)$. We say that the operator $F$ is \textbf{upper-triangular} if the operator $F$ admits the decomposition $F=\bigoplus_{\bm{\alpha}\in\mbb{N}^I}F_{\bm{\alpha}}$ such that $F_{\bm{\alpha}}:V_{\mbf{n}_1}\otimes W_{\mbf{n}_2}\rightarrow V_{\mbf{n}_1+\bm{\alpha}}\otimes W_{\mbf{n}_2-\bm{\alpha}}$ and $F_{\mbf{0}}=\text{Id}$. Similarly, we say that the operator $F$ is \textbf{lower-triangular} if the operator $F$ admits the decomposition $F=\bigoplus_{\bm{\alpha}\in\mbb{N}^I}F_{-\bm{\alpha}}$ such that $F_{\bm{0}}=\text{Id}$ and $F_{-\bm{\alpha}}:V_{\mbf{n}_1}\otimes W_{\mbf{n}_2}\rightarrow V_{\mbf{n}_1-\bm{\alpha}}\otimes W_{\mbf{n}_2+\bm{\alpha}}$.

The first step of the construction is the invertible solution to the ABRR equation. It was first introduced in \cite{ABRR97} and the following analogue was also formulated in \cite{OS22}.

\begin{prop}\label{ABRR-equation}
There exists an unique invertible strictly upper triangular operator $J_{\mbf{m}}^{+}(z)$， which we call it the fusion operator, and an unique invertible strictly lower triangular operator $J_{\mbf{m}}^{-}(z)$ solutions of the following ABRR equations:
\begin{align}\label{ABRR-eqn}
R_{\mbf{m}}^+J_{\mbf{m}}^+(z)z_{(1)}^{-1}q^{\Omega}=z_{(1)}^{-1}q^{\Omega}J_{\mbf{m}}^+(z),\qquad z_{(1)}^{-1}q^{\Omega}J_{\mbf{m}}^-(z)R_{\mbf{m}}^-=J_{\mbf{m}}^-(z)z_{(1)}^{-1}q^{\Omega}
\end{align}
where $R_{\mbf{m}}^{-}:=(R_{\mbf{m}}^+)_{21}$ is the transposition of $R_{\mbf{m}}^+$. Moreover, $J_{\mbf{m}}^{\pm}(z)$ are elements in a completion $\mc{B}_{\mbf{m}}\hat{\otimes}\mc{B}_{\mbf{m}}$ satisfying:
\begin{align*}
S_{\mbf{m}}\otimes S_{\mbf{m}}((J_{\mbf{m}}^{-}(z))_{21})=J_{\mbf{m}}^{+}(z).
\end{align*}
\end{prop}
\begin{proof}
Here we write the second ABRR equation as:
\begin{align*}
\text{Ad}_{z_{(1)}q^{-\Omega}}(J_{\mbf{m}}^-(z))=J_{\mbf{m}}^-(z)R_{\mbf{m}}^{-}
\end{align*}

Now since the slope subalgebra universal $R$-matrix $R_{\mbf{m}}^-$ is upper-triangular,  the fusion operator $J_{\mbf{m}}^-(z)$ will also be upper-triangular.

Now given a decomposition of the universal $R$-matrix and the fusion operator as:
\begin{align*}
J_{\mbf{m}}^{-}(z)=1+\sum_{\mbf{n}\in\mbb{N}^I}J_{\mbf{m},\mbf{n}}^-(z),\qquad R_{\mbf{m}}^-=1+\sum_{\mbf{n}\in\mbb{N}^I}R_{\mbf{m},\mbf{n}}^-
\end{align*}
Thus the $\mbf{n}$-th component of the ABRR equation can be written as:
\begin{align*}
J_{\mbf{m},\mbf{n}}^-(z)z^{-\mbf{n}}q^{l}=J_{\mbf{m},\mbf{n}}^-(z)+\sum_{\substack{\mbf{n}_1+\mbf{n}_2=\mbf{n}\\\mbf{n}_1<\mbf{n}}}J_{\mbf{m},\mbf{n}_1}^-(z)R_{\mbf{m},\mbf{n}_2}^-
\end{align*}

Hence
\begin{equation*}
J_{\mbf{m},\mbf{n}}^-(z)=\frac{1}{z^{-\mbf{n}}q^l-1}\sum_{\substack{\mbf{n}_1+\mbf{n}_2=\mbf{n}\\\mbf{n}_1<\mbf{n}}}J_{\mbf{m},\mbf{n}_1}^-(z)R_{\mbf{m},\mbf{n}_2}^-.
\end{equation*}

Thus $J_{\mbf{m},\mbf{n}}^-(z)$ is uniquely determined. For the invertibility, Note that the universal $R$-matrix $R_{\mbf{m}}^{\pm}$ is invertible and the inverse $(J_{\mbf{m}}^{\pm})^{-1}$ can be made into the solution of the following equation:
\begin{equation*}
z_{(1)}^{-1}q^{\Omega}(J_{\mbf{m}}^+(z))^{-1}=(J_{\mbf{m}}^+(z))^{-1}(R_{\mbf{m}}^{+})^{-1}z_{(1)}^{-1}q^{\Omega},\qquad(J_{\mbf{m}}^-(z))^{-1}z_{(1)}^{-1}q^{\Omega}=z_{(1)}^{-1}q^{\Omega}(R_{\mbf{m}}^-)^{-1}(J_{\mbf{m}}^-(z))^{-1}
\end{equation*}
and mimicking the above procedure, the solution $(J_{\mbf{m}}^{\pm}(z))^{-1}$ are of the same strictly upper(lower)-triangular type and also unique, which is just the inverse of the operator $J_{\mbf{m}}^{\pm}(z)$ respectively.

\end{proof}

We now consider the universal $R$-matrix $R_{\mbf{m}}^{\pm}(a):=(D_{a}\otimes\text{Id})R_{\mbf{m}}^{\pm}$ with the equivariant parametre $a$ as defined in Section \ref{sub:equivariant_parametres_and_automorphisms}. We denote the difference operator $T_a$ with respect to $a$ as:
\begin{align}\label{difference-operator-for-equivariant-variables}
T_af(a)=f(pa).
\end{align}

Let us consider the shifted version for the fusion operator:
\begin{align*}
\mbf{J}_{\mbf{m}}^{\pm}(z,a)=(D_a\otimes\text{Id})J_{\mbf{m}}^{\pm}(zp^{\mbf{m}}).
\end{align*}

\begin{prop}\label{wall-KZ:proposition}
$\mbf{J}_{\mbf{m}}^{\pm}(z,a)$ is the unique upper/lower-triangular solution to the following \textbf{wall Knizhnik-Zamolodchikov equations}:
\begin{align}\label{wall-KZ-equation}
T_a(R_{\mbf{m}}^+(a)\mbf{J}_{\mbf{m}}^+(z,a))z_{(1)}^{-1}q^{\Omega}=z_{(1)}^{-1}q^{\Omega}\mbf{J}_{\mbf{m}}^+(z,a),\qquad z_{(1)}^{-1}q^{\Omega}T_a(\mbf{J}_{\mbf{m}}^-(z,a)R_{\mbf{m}}^-(a))=\mbf{J}_{\mbf{m}}^-(z,a)z_{(1)}^{-1}q^{\Omega}.
\end{align}
\end{prop}

\begin{proof}
The second equation is equivalent to the following equation:
\begin{align*}
\text{Ad}_{z_{(1)}q^{-\Omega}}\mbf{J}_{\mbf{m}}^-(z,a)=\mbf{J}^{-}_{\mbf{m}}(z,pa)R_{\mbf{m}}^-(pa).
\end{align*}

Let us look for its solution, we still do the decomposition:
\begin{align*}
\mbf{J}_{\mbf{m}}^-(z,a)=1+\sum_{\mbf{n}}\mbf{J}_{\mbf{m},\mbf{n}}^-(z)a^{\mbf{m}\cdot\mbf{n}},\qquad R_{\mbf{m}}^-(pa)=1+\sum_{\mbf{n}}R_{\mbf{m},\mbf{n}}^-p^{\mbf{m}\cdot\mbf{n}}a^{\mbf{m}\cdot\mbf{n}}.
\end{align*}

By the ABRR equation in \ref{ABRR-equation} we have that:
\begin{align*}
\mbf{J}_{\mbf{m},\mbf{n}}^-(z)z^{-\mbf{n}}q^{l}=\mbf{J}_{\mbf{m},\mbf{n}}^-(z)p^{\mbf{m}\cdot\mbf{n}}+p^{\mbf{m}\cdot\mbf{n}}\sum_{\substack{\mbf{n}_1+\mbf{n}_2=\mbf{n}\\\mbf{n}_1<\mbf{n}}}\mbf{J}_{\mbf{m},\mbf{n}_1}^-(z)R_{\mbf{m},\mbf{n}_2}^-.
\end{align*}

Hence
\begin{align*}
\mbf{J}_{\mbf{m},\mbf{n}}^-(z)=\frac{1}{z^{-\mbf{n}}q^lp^{-\mbf{m}\cdot\mbf{n}}-1}\sum_{\substack{\mbf{n}_1+\mbf{n}_2=\mbf{n}\\\mbf{n}_1<\mbf{n}}}\mbf{J}_{\mbf{m},\mbf{n}_1}^-(z)R_{\mbf{m},\mbf{n}_2}^-.
\end{align*}

Thus we found that:
\begin{align*}
\mbf{J}_{\mbf{m}}^-(z,a)=(D_a\otimes\text{Id})J_{\mbf{m}}^-(zp^{\mbf{m}}).
\end{align*}

Similar computation shows that $\mbf{J}_{\mbf{m}}^+(z,a)$ is the solution for the first equation in \eqref{wall-KZ-equation}
\end{proof}

\subsection{Quantum lattice operator}
Following the construction in \cite{OS22}, fix the slope $\mbf{m}\in\mbb{R}^n$, we define the \textbf{dynamical monodromy operator} $\mbf{B}_{\mbf{m}}(z)$ as:
\begin{align}\label{defn-of-monodromy}
\mbf{B}_{\mbf{m}}(z):=\textbf{m}((1\otimes S_{\mbf{m}})(\mbf{J}_{\mbf{m}}^{+}(z)^{-1}))|_{z\mapsto zq^{\kappa}}
\end{align}
it lives in the completion of $\mc{B}_{\mbf{m}}(z_1,\cdots,z_I)$. Here $q^{\kappa}$ is defined by the following: When the operator $\kappa$ acts on $V_{\bm{\lambda}\bm{\gamma}^{-\mbf{n}}}$, it can be represented as a linear functional on $\mbb{N}^I$ such that:
\begin{align}\label{kappa-linear-functional}
\kappa=\bm{\lambda}-\Pi\mbf{n},\qquad\Pi=(\pi_{ij})_{i,j\in I}
\end{align}
and here $\Pi=(\pi_{ij})_{i,j\in I}$ is the matrix defined in \eqref{renewed-cartan-bialgebra-pairing}.

Now we fix a point $s\in\mbb{R}^n$ and an integer point $\mc{L}\in\mbb{Z}^n$, we define the \textbf{quantum lattice operator} $\mbf{B}^s_{\mc{L}}(z)$ starting from $s$ as:
\begin{align}\label{defn-of-quantum-lattice-operator}
\mbf{B}^s_{\mc{L}}(z)=\tau_{\mc{L}}(\prod_{\mbf{m}\in[s,s-\mc{L})}^{\rightarrow}\mbf{B}_{\mbf{m}}(z))
\end{align}
The quantum lattice operator lives in the completion of $\mbf{U}_{Q}(z_1,\cdots,z_{I})$. 
Given each weight subspace $V_{\mbf{n}}\subset V$ of a representation $V\in\mc{O}$ in the category $\mc{O}$. The quantum difference operator $\mbf{B}_{\mc{L}}^s(z)$ preserves the weight $V_{\mbf{n}}$. It turns out that the well-definedness of the quantum lattice operator \eqref{defn-of-quantum-lattice-operator} by the following proposition:

\begin{prop}\label{finite-product-result:proposition}
Fix the weighted subspace $V_{\mbf{n}}$ for a module $V$ in the category $\mc{O}$. The action of the quantum lattice operator $\mbf{B}_{\mc{L}}^s(z)$  on $V_{\mbf{n}}$ is reduced to the finite product of the dynamical monodromy operator $\mbf{B}_{\mbf{m}}(z)$ for $\mbf{m}$ on the arbitrary cutty-corner curve between $s$ and $s-\mc{L}$.
\end{prop}
\begin{proof}
By definition one can see that the dynamical monodromy operator $\mbf{B}_{\mbf{m}}(z)$ is written in the following way:
\begin{align*}
\mbf{B}_{\mbf{m}}(z)=\sum_{\mbf{k}\in\mbb{N}^I}^{\mbf{k}\cdot\mbf{m}\in\mbb{Z}}a_{\mbf{k}}(z)E_{\mbf{k}}E_{-\mbf{k}}
\end{align*}
which maps $V_{\mbf{n}}$ to $V_{\mbf{n}-\mbf{k}}$ to $V_{\mbf{n}}$ for some $\mbf{k}\in\mbb{N}^I$. Moreover, if $\mbf{k}>\mbf{n}$, and the corresponding operator in $\mbf{B}_{\mbf{m}}(z)$ is zero.

Now we consider the set of points $M_{\mbf{n}}$ of $\mbf{m}\in\mbb{R}^I$ such that $\mbf{n}'\cdot\mbf{m}\in\mbb{Z}$ for some $\mbf{n}'\leq\mbf{n}$, and we can see that this set is a union of hyperplanes in $\mbb{R}^I$. Given a curve of finite length, it is easy to see the curve only intersects with finitely many hyperplanes.

Although the quantum difference operator $\mbf{B}_{\mc{L}}^s$ might be infinite product of the dynamical monodromy operator $\mbf{B}_{\mbf{m}}(z)$, we can see that only the dynamical monodromy operator $\mbf{B}_{\mbf{m}}(z)$ that satisfies the property that the corresponding slope subalgebra $\mc{B}_{\mbf{m}}$ contains the elements $E_{-\mbf{k}}$ with $\mbf{k}\leq\mbf{n}$ and $\mbf{k}\cdot\mbf{m}\in\mbb{Z}$ would give nontrivial action on $V_{\mbf{n}}$. This requires that $\mbf{m}$ should lie in $M_{\mbf{n}}$. Since we are choosing the dynamical monodromy operator $\mbf{B}_{\mbf{m}}(z)$ of the slope $\mbf{m}$ on a cutty-corner curve from $s$ to $s-\mc{L}$, which is a curve of finite length. Thus from the above argument, there are only finitely many $\mbf{m}$ which gives nontrivial action.
\end{proof}

\subsection{Basic properties for the dynamical monodromy operators}

Here we give the basic properties for the dynamical monodromy operator. The argument and the statement closely follow \cite{OS22}.

We define $J_{\mbf{m}}^{\pm}(z)^{12}=J_{\mbf{m}}^{\pm}(z)\otimes1$, $J_{\mbf{m}}^{\pm}(z)^{23}=1\otimes J_{\mbf{m}}^{\pm}(z)$, $J_{\mbf{m}}^{\pm}(z)^{12,3}=(\Delta_{\mbf{m}}\otimes1)J_{\mbf{m}}^{\pm}(z)$, $J_{\mbf{m}}^{\pm}(z)^{1,23}=(1\otimes\Delta_{\mbf{m}})J_{\mbf{m}}^{\pm}(z)$.

\begin{prop}
The operator $J_{\mbf{m}}^{\pm}(z)$ satisfies the dynamical cocycle conditions:
\begin{equation*}
\begin{aligned}
&J_{\mbf{m}}^{+}(z)^{12,3}J_{\mbf{m}}^{+}(zq^{\kappa_{(3)}})^{12}=J_{\mbf{m}}^{+}(z)^{1,23}J_{\mbf{m}}^{+}(zq^{-\kappa_{(1)}})^{23}\\
&J_{\mbf{m}}^{-}(zq^{\kappa_{(3)}})^{12}J_{\mbf{m}}^{-}(z)^{12,3}=J_{\mbf{m}}^{-}(zq^{-\kappa_{(1)}})^{23}J_{\mbf{m}}^{-}(z)^{1,23}
\end{aligned}
\end{equation*}
\end{prop}
\begin{proof}
The proof is similar to the Theorem $5$ in \cite{OS22}. First note that there is a unique strictly upper-triangular operator $J(z)$ such that:
\begin{equation*}
\begin{aligned}
&J(z)z_{(3)}(q^{\Omega}(R_{\mbf{m}}^{-})^{-1})^{13}(q^{\Omega}(R_{\mbf{m}}^{-})^{-1})^{23}=z_{(3)}q^{\Omega_{13}+\Omega_{23}}J(z)\\
&J(z)z_{(1)}^{-1}(q^{\Omega}(R_{\mbf{m}}^{-})^{-1})^{13}(q^{\Omega}(R_{\mbf{m}}^{-})^{-1})^{12}=z_{(1)}^{-1}q^{\Omega_{13}+\Omega_{12}}J(z).
\end{aligned}
\end{equation*}

We claim that the operator:
\begin{align*}
J(z)=J^-_{\mbf{m}}(zq^{\kappa_{(3)}})^{12}J^-_{\mbf{m}}(z)^{12,3}
\end{align*}
is the solution to the above equation.

We will denote the operator:
\begin{align*}
&A_{R}(X)=q^{-\Omega_{13}-\Omega_{23}}z_{(3)}^{-1}Xz_{(3)}(R_{\mbf{m}}^{-,13})^{-1}(R_{\mbf{m}}^{-,23})^{-1}\\
&A_{L}(X)=q^{-\Omega_{13}-\Omega_{12}}z_{(1)}Xz_{(1)}^{-1}(R_{\mbf{m}}^{-,13})^{-1}(R_{\mbf{m}}^{-,12})^{-1}.
\end{align*}

As in the ABRR equation, the solution is uniquely determined by the degree zero part in the third tensor component. Therefore let us denote the corresponding component of $J(z)$ as $J_0(z)$. It is easy to see that:
\begin{align*}
J_0(z)=J^-_{\mbf{m}}(zq^{\kappa_{(3)}})^{12}
\end{align*}
and we denote the operator:
\begin{align*}
Y_0(z)=q^{-\Omega_{13}-\Omega_{12}}z_{(1)}J^-_{\mbf{m}}(zq^{\kappa_{(3)}})^{12}z_{(1)}^{-1}(R_{\mbf{m}}^{-,13})^{-1}(R_{\mbf{m}}^{-,12})^{-1}
\end{align*}

First note that $J_0(z)=1+\sum_{\mbf{n}}J_{0,\mbf{n}}(z)$ has a weight decomposition. Also note that:
\begin{align*}
q^{-\Omega_{13}}J_{0,\mbf{n}}(z)q^{\Omega_{13}}=q^{m_{\mbf{n}}}J_{0,\mbf{n}}(z),\qquad m_{\mbf{n}}=\langle\mbf{n},\kappa\rangle
\end{align*}
This follows from the following computation:
Consider the tensor product of three simple modules $L(\bm{\psi}_1)\otimes L(\bm{\psi}_2)\otimes L(\bm{\psi}_3)$. If we have the operator of the following form:
\begin{align*}
Z_{\alpha}:L(\bm{\psi}_1)_{\mbf{n}_1}\otimes L(\bm{\psi}_2)_{\mbf{n}_2}\otimes L(\bm{\psi}_3)_{\mbf{n}_3}\rightarrow L(\bm{\psi}_1)_{\mbf{n}_1+\alpha}\otimes L(\bm{\psi}_2)_{\mbf{n}_2-\alpha}\otimes L(\bm{\psi}_3)_{\mbf{n}_3}
\end{align*}
In this we have that:
\begin{align*}
q^{-\Omega_{13}}Z_{\alpha}q^{\Omega_{13}}=q^{\langle\kappa,\alpha\rangle}Z_{\alpha}.
\end{align*}

Since by the computation we have that:
\begin{equation*}
\begin{aligned}
&\Omega(\mbf{m},\mbf{n})-\Omega(\mbf{m}+\alpha,\mbf{n})\\
=&\hbar\sum_{i,j\in I}(\pi^{-1})_{ij}\sum_{k,l\in I}[(m_k\pi_{ik}+\lambda_{1i})(n_l\pi_{jl}+\lambda_{2j})-((m_k+\alpha_k)\pi_{ik}+\lambda_{1i})(n_l\pi_{jl}+\lambda_{2j})]\\
=&-\hbar\sum_{i,j\in I}(\pi^{-1})_{ij}\sum_{k,l\in I}\alpha_k\pi_{ik}(n_l\pi_{jl}+\lambda_{2j})\\
=&-\hbar\sum_{i,j\in I}\sum_{k,l\in I}(\pi^{-1})_{ij}\pi_{ik}\pi_{jl}\alpha_kn_l-\hbar\sum_{i,j\in I}\sum_{k,l\in I}(\pi^{-1})_{ij}\alpha_k\pi_{ik}\lambda_{2j}\\
=&-\hbar\sum_{i,k,l\in I}\delta_{il}\pi_{ik}\alpha_kn_l+\hbar\sum_{j,k,l\in I}\delta_{jk}\alpha_k\lambda_{2j}\\
=&-\hbar\sum_{k,l\in I}\pi_{lk}\alpha_kn_l+\hbar\sum_{k,l}\alpha_k\lambda_{2l}=\hbar(\langle\alpha,\kappa\rangle)
\end{aligned}
\end{equation*}
with $\kappa$ defined in \eqref{kappa-linear-functional}.

Thus we have that:
\begin{align*}
Y_0(z)=q^{-\Omega_{12}}z_{(1)}q^{\kappa_{(3)}}J_{\mbf{m}}^+(zq^{\kappa_{(3)}})^{12}z_{(1)}^{-1}q^{-\kappa_{(3)}}(R_{\mbf{m}}^{-,12})^{-1}.
\end{align*}
which is an equivalent form of the ABRR equation. Hence we have $Y_0(z)=J_0(z)$.

\end{proof}

Now we consider the operator:
\begin{align*}
\tilde{B}_{\mbf{m}}(z):=\mbf{m}((1\otimes S_{\mbf{m}})(J_{\mbf{m}}^+(z)^{-1}))|_{z\mapsto zq^{\kappa}},\qquad B_{\mbf{m}}(z)=\mbf{m}_{21}((S_{\mbf{m}}^{-1}\otimes1)(J_{\mbf{m}}^+(z)^{-1}))|_{z\mapsto zq^{-\kappa}}.
\end{align*}

We have the following theorem on their coproduct:
\begin{thm}
\begin{equation*}
\begin{aligned}
&\Delta_{\mbf{m}}(\tilde{B}_{\mbf{m}}(z))=J_{\mbf{m}}^+(z)(\tilde{B}_{\mbf{m}}(zq^{\kappa_{(2)}})\otimes\tilde{B}_{\mbf{m}}(zq^{-\kappa_{(1)}}))J_{\mbf{m}}^-(z)\\
&\Delta_{\mbf{m}}(B_{\mbf{m}}(z))=J_{\mbf{m}}^+(z)(B_{\mbf{m}}(zq^{\kappa_{(2)}})\otimes B_{\mbf{m}}(zq^{-\kappa_{(1)}}))J_{\mbf{m}}^-(z).
\end{aligned}
\end{equation*}
\end{thm}

\begin{proof}
The proof is given in Theorem 6 in \cite{OS22}, which is totally the same.
\end{proof}

One of the important results is that the coproduct for the dynamical monodromy operators are given by:
\begin{align}\label{coproduct-dynamical-monodromy}
\Delta_{\mbf{m}}(\mbf{B}_{\mbf{m}}(z))=\mbf{J}_{\mbf{m}}^+(z)(\mbf{B}_{\mbf{m}}(zq^{\kappa_{(2)}})\otimes\mbf{B}_{\mbf{m}}(zq^{-\kappa_{(1)}}))\mbf{J}_{\mbf{m}}^-(z)
\end{align}
and the proof is mimicking the proof given in Corollary 3 of \cite{OS22}

\subsection{Properties for the quantum difference operators}
Now let us consider the qKZ operator for the slope subalgebras, and we shall denote:
\begin{align*}
\mc{K}_{\mbf{m}}^-=T_{a}z_{(1)}^{-1}q^{\Omega}(R_{\mbf{m}}^-(a))^{-1},\qquad \mc{K}_{\mbf{m}}^+=q^{\Omega}(R_{\mbf{m}}^+(a))^{-1}T_az_{(1)}^{-1}.
\end{align*}

\begin{prop}\label{wall-qKZ-and-coproduct-monodromy}
\begin{align*}
\mc{K}_{\mbf{m}}^+\Delta_{\mbf{m}}(\mbf{B}_{\mbf{m}}(z))=\Delta_{\mbf{m}}(\mbf{B}_{\mbf{m}}(z))\mc{K}_{\mbf{m}}^-
\end{align*}
\end{prop}
\begin{proof}
\begin{equation*}
\begin{aligned}
\mc{K}_{\mbf{m}}^+\Delta_{\mbf{m}}(\mbf{B}_{\mbf{m}}(z))=&q^{\Omega}(R_{\mbf{m}}^+(a))^{-1}T_az_{(1)}^{-1}\mbf{J}_{\mbf{m}}^+(z)(\mbf{B}_{\mbf{m}}(zq^{\kappa_{(2)}})\otimes\mbf{B}_{\mbf{m}}(zq^{-\kappa_{(1)}}))\mbf{J}_{\mbf{m}}^-(z)\\
=&\mbf{J}_{\mbf{m}}^+(z)T_az_{(1)}^{-1}q^{\Omega}(\mbf{B}_{\mbf{m}}(zq^{\kappa_{(2)}})\otimes\mbf{B}_{\mbf{m}}(zq^{-\kappa_{(1)}}))\mbf{J}_{\mbf{m}}^{-}(z)\\
=&\mbf{J}_{\mbf{m}}^+(z)(\mbf{B}_{\mbf{m}}(zq^{\kappa_{(2)}})\otimes\mbf{B}_{\mbf{m}}(zq^{-\kappa_{(1)}}))T_az_{(1)}^{-1}q^{\Omega}\mbf{J}_{\mbf{m}}^{-}(z)\\
=&\mbf{J}_{\mbf{m}}^+(z)(\mbf{B}_{\mbf{m}}(zq^{\kappa_{(2)}})\otimes\mbf{B}_{\mbf{m}}(zq^{-\kappa_{(1)}}))\mbf{J}_{\mbf{m}}^{-}(z)z_{(1)}^{-1}T_aq^{\Omega}(R_{\mbf{m}}^-(a))^{-1}\\
=&\Delta_{\mbf{m}}(\mbf{B}_{\mbf{m}}(z))\mc{K}_{\mbf{m}}^-
\end{aligned}
\end{equation*}

Here the second and the fourth equalities come from the wall KZ equation \eqref{wall-KZ-equation}. Hence the proof is finished.
\end{proof}

For $\mc{L}=(k_1,\cdots,k_I)\in\mbb{Z}^I$, we denote the difference operator $T_{\mc{L}}$ for the Kahler variables $(z_i)_{i\in I}$ as:
\begin{align*}
T_{\mc{L}}F(z_1,\cdots,z_{I})=F(z_1p^{k_1},\cdots,z_{I}p^{k_I})
\end{align*}

\begin{prop}\label{translation-formula}
Fix $\mc{L}\in\mbb{Z}^I$, we have the following translation symmetry formula:
\begin{align*}
T_{\mc{L}}^{-1}\mbf{B}_{\mbf{m}}(z)=\tau_{\mc{L}^{-1}}\mbf{B}_{\mbf{m}+\mc{L}}(z)T_{\mc{L}}^{-1}.
\end{align*}
\end{prop}
\begin{proof}
Since the monodromy operator $\mbf{B}_{\mbf{m}}(z)$ is defined by the fusion operator in \ref{defn-of-monodromy}, and the fusion operator $\mbf{J}_{\mbf{m}}^{\pm}(z)$ is totally determined by the ABRR equation \ref{ABRR-eqn}.
We can see that the proposition follows from the following identity:
\begin{align*}
\tau_{\mc{L}\otimes\mc{L}}R_{\mbf{m}}^{\pm}=R_{\mbf{m}+\mc{L}}^{\pm}.
\end{align*}

We know that the adjoint action by $\mc{L}$ induces the isomorphism of slope subalgebras:
\begin{align*}
\tau_{\mc{L}}:\mc{B}_{\mbf{m}}\cong\mc{B}_{\mbf{m}+\mc{L}}
\end{align*}

We further need to prove the following equality:
\begin{align*}
\tau_{\mc{L}\otimes\mc{L}}(\Delta_{\mbf{m}}(F))=\Delta_{\mbf{m}+\mc{L}}(\tau_{\mc{L}}(F)),\qquad \tau_{\mc{L}\otimes\mc{L}}(\Delta_{\mbf{m}}(G))=\Delta_{\mbf{m}+\mc{L}}(\tau_{\mc{L}}(G))
\end{align*}
for $F\in\mbf{U}_{Q}^+$ and $G\in\mbf{U}_{Q}^-$. For simplicity we only check the case $F\in\mbf{U}_{Q}^+$ since the proof for the negative half is totally the same.
This can be checked by the following direct computation:
\begin{equation*}
\begin{aligned}
&\tau_{\mc{L}\otimes\mc{L}}(\Delta_{\mbf{m}}(F))=\tau_{\mc{L}\otimes\mc{L}}(\sum_{\mbf{0}\leq\mbf{k}\leq\mbf{n}}\lim_{\xi\rightarrow\infty}\frac{h_{\mbf{n}-\mbf{k}}F(\cdots,z_{i1},\cdots,z_{ik_i}\otimes\xi z_{i,k_{i}+1},\cdots,\xi z_{in_i},\cdots)}{\xi^{\mbf{m}\cdot(\mbf{n}-\mbf{k})}\text{lead}[\prod_{1\leq a\leq k_i}^{i\in I}\prod_{k_j<b\leq n_j}^{j\in I}\zeta_{ji}(\xi z_{jb}/z_{ia})]})\\
=&(\sum_{\mbf{0}\leq\mbf{k}\leq\mbf{n}}\lim_{\xi\rightarrow\infty}\frac{h_{\mbf{n}-\mbf{k}}F(\cdots,z_{i1},\cdots,z_{ik_i}\otimes\xi z_{i,k_{i}+1},\cdots,\xi z_{in_i},\cdots)}{\xi^{\mbf{m}\cdot(\mbf{n}-\mbf{k})}\text{lead}[\prod_{1\leq a\leq k_i}^{i\in I}\prod_{k_j<b\leq n_j}^{j\in I}\zeta_{ji}(\xi z_{jb}/z_{ia})]}\prod_{1\leq l_i\leq k_i}z_{il_i}\otimes\prod_{k_i+1\leq l_i\leq n_i}z_{il_i})\\
=&(\sum_{\mbf{0}\leq\mbf{k}\leq\mbf{n}}\lim_{\xi\rightarrow\infty}\frac{h_{\mbf{n}-\mbf{k}}F(\cdots,z_{i1},\cdots,z_{ik_i}\otimes\xi z_{i,k_{i}+1},\cdots,\xi z_{in_i},\cdots)}{\xi^{(\mbf{m}+\mc{L})\cdot(\mbf{n}-\mbf{k})}\text{lead}[\prod_{1\leq a\leq k_i}^{i\in I}\prod_{k_j<b\leq n_j}^{j\in I}\zeta_{ji}(\xi z_{jb}/z_{ia})]}\prod_{1\leq l_i\leq k_i}z_{il_i}\otimes\prod_{k_i+1\leq l_i\leq n_i}\xi z_{il_i})\\
=&\Delta_{\mbf{m}+\mc{L}}(\tau_{\mc{L}}(F))
\end{aligned}
\end{equation*}

Using the similar argument, one can show that:
\begin{align*}
\tau_{\mc{L}\otimes\mc{L}}(\Delta_{\mbf{m}}^{op}(F))=\Delta_{\mbf{m}+\mc{L}}^{op}(\tau_{\mc{L}}(F))
\end{align*}
Using the formula that:
\begin{align*}
\Delta_{\mbf{m}}^{op}=q^{\Omega}R_{\mbf{m}}^+\Delta_{\mbf{m}}(q^{\Omega}R_{\mbf{m}}^{+})^{-1}
\end{align*}
we have
\begin{equation*}
\begin{aligned}
\tau_{\mc{L}\otimes\mc{L}}(\Delta_{\mbf{m}}^{op}(F))=&\tau_{\mc{L}\otimes\mc{L}}(q^{\Omega}R_{\mbf{m}}^+)\tau_{\mc{L}\otimes\mc{L}}(\Delta_{\mbf{m}}(F))\tau_{\mc{L}\otimes\mc{L}}(q^{\Omega}R_{\mbf{m}}^+)^{-1}\\
=&q^{\Omega}R_{\mbf{m}+\mc{L}}^+\Delta_{\mbf{m}+\mc{L}}(\tau_{\mc{L}}(F))(q^{\Omega}R_{\mbf{m}+\mc{L}}^+)^{-1}\\
=&q^{\Omega}R_{\mbf{m}+\mc{L}}^+\tau_{\mc{L}\otimes\mc{L}}(\Delta_{\mbf{m}}(F))(q^{\Omega}R_{\mbf{m}+\mc{L}}^+)^{-1}.
\end{aligned}
\end{equation*}

Thus it follows that both $q^{\Omega}R_{\mbf{m}+\mc{L}}^+$ and $\tau_{\mc{L}\otimes\mc{L}}(q^{\Omega}R_{\mbf{m}}^+)$ are the universal $R$-matrices for the slope subalgebra $\mc{B}_{\mbf{m}+\mc{L}}$. To prove that they are the same, we need to prove that $\tau_{\mc{L}\otimes\mc{L}}(q^{\Omega}R_{\mbf{m}}^+)$ also comes from the Drinfeld pairing on $\mc{B}_{\mbf{m}+\mc{L}}$. It is easily checked that $\tau_{\mc{L}\otimes\mc{L}}(q^{\Omega}R_{\mbf{m}}^+)$ satisfies the Yang-Baxter equation. Since the pairing is non-degenerate, we denote the basis for the universal $R$-matrices of $\mc{B}_{\mbf{m}}$ and $\mc{B}_{\mbf{m}+\mc{L}}$ as:
\begin{align*}
R_{\mbf{m}}^+=\sum_{i}F_{i}^{(\mbf{m})}\otimes G_{i}^{(\mbf{m})},\qquad R_{\mbf{m}+\mc{L}}^+=\sum_{i}F_{i}^{(\mbf{m}+\mc{L})}\otimes G_{i}^{(\mbf{m}+\mc{L})}.
\end{align*}

Note that the expression of both $R_{\mbf{m}}^+$ and $R_{\mbf{m}+\mc{L}}^+$ is independent of the choice of the basis. Moreover, $\tau_{\mc{L}}$ is an isomorphism for each graded pieces $\mc{B}_{\mbf{m},\mbf{k}}\cong\mc{B}_{\mbf{m}+\mc{L},\mbf{k}}$. Thus we only need to prove the following thing:
\begin{lem}
Given $F\in\mc{B}_{\mbf{m}}^+$ and $G\in\mc{B}_{\mbf{m}}^{-}$ such that $\langle F,G\rangle=0$. We have that:
\begin{align*}
\langle\tau_{\mc{L}}F,\tau_{\mc{L}}G\rangle=0
\end{align*}
\end{lem}
\begin{proof}
This is from the definition of the shift automorphism $\tau_{\mc{L}}$ in \eqref{shift-automorphism-quantum-positive} and \eqref{shift-automorphism-quantum-negative}, and the definition of the bialgebra pairing \eqref{nondegenerate-bialgebra-pairing}.
\end{proof}

Since each graded pieces $\mc{B}_{\mbf{m},\mbf{k}}$ are finite-dimensional, we conclude that $\tau_{\mc{L}}(F_{i}^{(\mbf{m})})$, $\tau_{\mc{L}}(G_{i}^{(\mbf{m})})$ is a set of orthogonal basis for $\mc{B}_{\mbf{m}}$. By the independence of the orthogonal basis for the universal $R$-matrix with respect to the nondegenerate bialgebra pairing, we conclude that
\begin{align*}
\tau_{\mc{L}\otimes\mc{L}}(R_{\mbf{m}}^+)=R_{\mbf{m}+\mc{L}}^+.
\end{align*}
Thus we can conclude the proposition.
\end{proof}

Now we describe the coproduct for the quantum lattice operator \eqref{defn-of-quantum-lattice-operator}.
\begin{prop}\label{coproduct-quantum-difference-operator}
Given two representation $V,W\in\mc{O}$ we have that:
\begin{align*}
\Delta_{\mbf{m}}(T_{\mc{L}}^{-1}\mbf{B}^s_{\mc{L}}(z))=W_{\mbf{m}_0}(z)W_{\mbf{m}_1}(z)\cdots W_{\mbf{m}_{n-1}}(z)\prod_{\mbf{m}\in[s,s+\mc{L})}(R_{\mbf{m}}^-)^{-1}T_{\mc{L}}^{-1}
\end{align*}
as the endomorphism of $V\otimes W$,
and here $\mbf{m}_0,\cdots,\mbf{m}_{n-1}$ are the slope points on the interval $s+\mc{L}$ and $s$. $W_{\mbf{m}}(z):=\Delta_{\mbf{m}}(\mbf{B}_{\mbf{m}}(z))R_{\mbf{m}}^-$.
\end{prop}
\begin{proof}
By Proposition \ref{finite-product-result:proposition}, the quantum lattice operator $\mbf{B}_{\mc{L}}^s(z)$ can be written as the finite product.
By definition we have that:
\begin{align*}
T_{\mc{L}}^{-1}\mbf{B}^s_{\mc{L}}(z)=T_{\mc{L}}^{-1}\tau_{\mc{L}}(\mbf{B}_{\mbf{m}_{-n}}(z)\cdots\mbf{B}_{\mbf{m}_{-1}}(z)).
\end{align*}

By the Proposition \ref{translation-formula} we have that:
\begin{align*}
\mbf{B}^s_{\mc{L}}(z)=\mbf{B}_{\mbf{m}_0}(z)\mbf{B}_{\mbf{m}_1}(z)\cdots\mbf{B}_{\mbf{m}_{n-1}}(z)T_{\mc{L}}^{-1},\qquad\mbf{m}_{k+n}=\mbf{m}_k+\mc{L}.
\end{align*}

Then for the coproduct we have
\begin{align*}
\Delta_s(\mbf{B}^s_{\mc{L}}(z))=\Delta_s(\mbf{B}_{\mbf{m}_0}(z)\mbf{B}_{\mbf{m}_1}(z)\cdots\mbf{B}_{\mbf{m}_{n-1}}(z))T_{\mc{L}}^{-1}.
\end{align*}

By the formula given in \eqref{Relation-between-coproducts-of-different-slope}:
\begin{align*}
\Delta_s(\mbf{B}_{\mbf{m}_k}(z))=(R_{\mbf{m}_0}^-)\cdots(R_{\mbf{m}_{k-1}}^-)\Delta_{\mbf{m}_k}(\mbf{B}_{\mbf{m}_k}(z))(R_{\mbf{m}_{k-1}}^-)^{-1}\cdots (R_{\mbf{m}_{0}}^-)^{-1}.
\end{align*}

Combining the result together we obtain the proposition.

\end{proof}

Now we define the qKZ operator of slope $s$ as:
\begin{align}\label{defn-of-qKZ-operator-1-point}
\mc{K}^s=z_{(1)}T_a^{-1}\mc{R}^s(a)
\end{align}
\begin{thm}\label{commuting-relation-for-qKZ-and-dynamical-monodromy:theorem}
Fix two representations $V,W\in\mc{O}$, and let $s$ and $s'$ be two slopes separated by a single slope $\mbf{m}$, we have that:
\begin{align*}
W_{\mbf{m}}(z)^{-1}\mc{K}^s W_{\mbf{m}}(z)=\mc{K}^{s'},\qquad W_{\mbf{m}}(z)^{-1}\Delta_{s}(\mbf{B}^s_{\mc{L}}(z))W_{\mbf{m}}(z)=\Delta_{s'}(\mbf{B}^{s'}_{\mc{L}}(z))
\end{align*}
as the endomorphism of $\bigoplus_{\mbf{n}_1+\mbf{n}_2=\mbf{n}}V_{\mbf{n}_1}\otimes W_{\mbf{n}_2}$ with the fixed $\mbf{n}\in\mbb{N}^I$.
\end{thm}
\begin{proof}
By computation:
\begin{equation*}
\begin{aligned}
\mc{K}^sW_{\mbf{m}}(z)=&z_{(1)}T_{a}^{-1}\mc{R}^s(a)\Delta_{\mbf{m}}(\mbf{B}_{\mbf{m}}(z))R_{\mbf{m}}^-\\
=&z_{(1)}T_a^{-1}\Delta_{\mbf{m}}^{op}(\mbf{B}_{\mbf{m}}(z))\mc{R}^s(a)R_{\mbf{m}}^-\\
=&z_{(1)}T_{a}^{-1}(R_{\mbf{m}}^+)q^{-\Omega}q^{\Omega}(R_{\mbf{m}}^+)^{-1}\Delta_{\mbf{m}}^{op}(\mbf{B}_{\mbf{m}}(z))\mc{R}^s(a)R_{\mbf{m}}^-\\
=&z_{(1)}T_{a}^{-1}R_{\mbf{m}}^+q^{-\Omega}\Delta_{\mbf{m}}(\mbf{B}_{\mbf{m}}(z))q^{\Omega}(R_{\mbf{m}}^{+})^{-1}\mc{R}^s(a)R_{\mbf{m}}^-\\
=&\Delta_{\mbf{m}}(\mbf{B}_{\mbf{m}}(z))(R_{\mbf{m}}^-)z_{(1)}^{-1}T_a^{-1}(R_{\mbf{m}}^+)^{-1}\mc{R}^s(a)R_{\mbf{m}}^-=W_{\mbf{m}}(z)\mc{K}^{s'}.
\end{aligned}
\end{equation*}

The second last equality comes from Proposition \ref{wall-qKZ-and-coproduct-monodromy}, and the last equality above comes from the fact on when restricted to $\bigoplus_{\mbf{n}_1+\mbf{n}_2=\mbf{n}}\text{End}(V_{\mbf{n}_1}\otimes W_{\mbf{n}_2})$, we have $(R_{\mbf{m}}^+)^{-1}\mc{R}^s(a)R_{\mbf{m}}^-=\mc{R}^{s'}(a)$.

For the second equality, by the Proposition \ref{coproduct-quantum-difference-operator} we have:
\begin{align*}
&\Delta_s(\mbf{B}_{\mc{L}}^s(z))=W_{\mbf{m}_0}(z)\cdots W_{\mbf{m}_{n-1}}(z)\prod_{\mbf{m}\in[s,s+\mc{L})}^{\rightarrow}(R_{\mbf{m}}^-)^{-1}T_{\mc{L}}^{-1}\\
&\Delta_{s'}(\mbf{B}_{\mc{L}}^{s'}(z))=W_{\mbf{m}_1}(z)\cdots W_{\mbf{m}_n}(z)\prod_{\mbf{m}\in[s',s'+\mc{L})}^{\rightarrow}(R_{\mbf{m}}^-)^{-1}T_{\mc{L}}^{-1}.
\end{align*}

Then the equality follows from the translation symmetry formula in Proposition \ref{translation-formula}.

\end{proof}

The coproduct of the quantum lattice operator has the following relation with the qKZ operator:
\begin{prop}
For arbitrary $\mc{L},\mc{L}'\in\mbb{Z}^I$ and the slope $s\in\mbb{R}^I$ we have the following commutation relations:
\begin{align*}
\mc{K}^s\Delta_s(\mbf{B}^s_{\mc{L}}(z))T_{\mc{L}}(\prod_{\mbf{m}\in[s,s+\mc{L})}^{\leftarrow}R_{\mbf{m}}^-)=\Delta_{s}(\mbf{B}^s_{\mc{L}}(z))T_{\mc{L}}(\prod_{\mbf{m}\in[s,s+\mc{L})}^{\leftarrow}R_{\mbf{m}}^-)\mc{K}^{s+\mc{L}}
\end{align*}
\end{prop}
\begin{proof}
The proof follows the following computation:
\begin{equation*}
\begin{aligned}
\mc{K}^s\Delta_s(\mbf{B}^s_{\mc{L}}(z))=&\mc{K}^sW_{\mbf{m}_0}(z)\cdots W_{\mbf{m}_{n-1}}(z)\prod_{\mbf{m}\in[s,s+\mc{L})}^{\rightarrow}(R_{\mbf{m}}^-)^{-1}T_{\mc{L}}^{-1}\\
=&W_{\mbf{m}_0}(z)\cdots W_{\mbf{m}_{n-1}}(z)\mc{K}^{s+\mc{L}}\prod_{\mbf{m}\in[s,s+\mc{L})}^{\rightarrow}(R_{\mbf{m}}^-)^{-1}T_{\mc{L}}^{-1}\\
=&
\Delta_{s}(\mbf{B}^s_{\mc{L}}(z))T_{\mc{L}}(\prod_{\mbf{m}\in[s,s+\mc{L})}^{\leftarrow}R_{\mbf{m}}^-)\mc{K}^{s+\mc{L}}\prod_{\mbf{m}\in[s,s+\mc{L})}^{\rightarrow}(R_{\mbf{m}}^-)^{-1}T_{\mc{L}}^{-1}
\end{aligned}
\end{equation*}
and here the second equality comes from Theorem \ref{commuting-relation-for-qKZ-and-dynamical-monodromy:theorem}.

\end{proof}

Now we come to the uniqueness of the quantum lattice operator:
\begin{thm}\label{path-independence-theorem}
The quantum lattice operator $\mbf{B}^s_{\mc{L}}(z)$ in \eqref{defn-of-quantum-lattice-operator} is independent of the choice of the path from $s$ to $s-\mc{L}$ as an operator in $\text{End}(V_{\mbf{n}})(z)$ for the weighted piece $V_{\mbf{n}}\subset V$ in the representation $V\in\mc{O}$.
\end{thm}

\begin{proof}
Now given two quantum lattice operators $\mbf{B}^s_{\mc{L}}(z), \mbf{B}'^s_{\mc{L}}(z)$ which represents two different paths from $s$ to $s-\mc{L}$, we denote their difference by:
\begin{align*}
D(z):=\mbf{B}'^s_{\mc{L}}(z)(\mbf{B}^s_{\mc{L}}(z))^{-1}.
\end{align*}

It satisfies the following relation:
\begin{align*}
\mc{K}^s\Delta_s(D)(z,ap)=\Delta_s(D)(z,ap)\mc{K}^s.
\end{align*}

By the commutativity with the qKZ operator, we have the following difference equation:
\begin{align*}
\Delta_{s}(D)(z,ap)=z_{(1)}R^s(a)\Delta_{s}(D)(z,a)(z_{(1)}R^s(a))^{-1}=z_{(1)}\Delta_{s}^{op}(D)(z,a)z_{(1)}^{-1}
\end{align*}

By definition we know that $\Delta_{s}(D)(z,a)$ is a Laurent polynomial in $a$. Moreover it has positive degree of $a$ in the block upper-triangular part, and it has negative degree of $a$ in the block lower-triangular part. By definition we also know that it can be written in the following way:
\begin{align*}
\Delta_{s}(D)(z,a)=\sum_{\mbf{k}}\frac{D_{\mbf{k}}(a)}{(p^{k_1}-a_{k_1})\cdots(p^{k_n}-a_{k_n})}
\end{align*}
Here $D_{\mbf{k}}(a)$ is a Laurent polynomial in $a$ and contains no formal parametre $p$, and $q_{(1)}^{\lambda}$ contains no formal parametre $p$ as well. Then since $q^{\lambda}_{(1)}$ is a diagonal operator. We have the following difference equation:
\begin{align*}
\Delta_{s}(D)(z,ap^2)=z_{(1)}^{2}\Delta_{s}(D)(z,a)z_{(1)}^{-2}
\end{align*}

When restricted to the weighted pieces, the difference equation looks like the following type:
\begin{align}\label{qde-for-D-on-a}
\Delta_{s}(D)(z,ap^2)=z^{(\cdots)}\Delta_{s}(D)(z,a)
\end{align}
Now since $\Delta_{s}(D)(z,a)$ is a Laurent polynomial in $a$, and moreover, it is a monomial of $a$ on each weighted pieces, this means that $\Delta_{s}(D(z,a))$ is constant on $a$. The difference equation \ref{qde-for-D-on-a} implies that $\Delta_{s}(D)(z,a)$ is a block-diagonal operator.
Thus $D_{\mbf{k}}(a)=D_{\mbf{k}}$ is independent of $a$, which implies that $\Delta_{s}(D)(z,a)=\Delta_{s}(D)(z)$ is block-diagonal.

Then by the coproduct formula we know that:
\begin{align*}
\Delta_{s}(D(z))=D(zq^{\kappa_{(2)}})\otimes D(zq^{-\kappa_{(1)}})
\end{align*}

Thus we have that:
\begin{align*}
(D(zq^{\kappa_{(2)}})\otimes D(zq^{-\kappa_{(1)}}))\mc{R}^s(a)=\mc{R}^s(a)D(zq^{\kappa_{(2)}})\otimes D(zq^{-\kappa_{(1)}})
\end{align*}

But we know that $\mc{R}^s(a)\Delta_s(D(z))(\mc{R}^s(a))^{-1}=\Delta_s^{op}(D(z))$, thus we have that:
\begin{align*}
D(zq^{\kappa_{(2)}})\otimes D(zq^{-\kappa_{(1)}})=D(zq^{-\kappa_{(1)}})\otimes D(zq^{\kappa_{(2)}}).
\end{align*}

Thus let us choose $V_{\bm{\lambda}_i}\otimes W_{\mbf{n}}$ such that $V_{\bm{\lambda}_i}$ is the lowest weight vector space in $V\in\mc{O}$. In this case we have the difference equation on Kahler variables $z$:
\begin{align*}
1\otimes D(zq^{-\kappa_{(1)}})=1\otimes D(zq^{\kappa_{(2)}})
\end{align*}

This implies that $D(zq^{\kappa})=D(z)$, which is $q$-periodic. But we know that $D(z)$ is a rational function over $z$, and this means that $D(z)=D$ is constant over $z$. By the definition of $D(z)$ and the quantum difference operators, this means that $D=\text{Id}$ is the identity operator.
\end{proof}

\subsection{Algebraic quantum difference equations}

Fix the slope point $s\in\mbb{R}^I$ and the weighted pieces $V_{\mbf{n}}\subset V\in\mc{O}$, and we further assume that $V\in\mc{O}$ is a tensor product of semisimple modules in $\mc{O}$. We define the algebraic quantum difference equation as:
\begin{align}\label{defn-qde-equation}
\Psi^s(zp^{\mc{L}})=\mbf{M}_{\mc{L}}^s(z)\Psi^s(z)\in\text{End}(V_{\mbf{n}})((z)),\qquad\mbf{M}_{\mc{L}}^s(z):=\mbf{B}^s_{\mc{L}}(z)\mc{L}
\end{align}
and here $\mc{L}$ stands for the line bundle operator constructed in Lemma \ref{line-bundle-operator:lemma} corresponding to the automorphism $\tau_{\mc{L}}$ for $\mc{L}\in\mbb{Z}^I$. From now on we fix our module $V\in\mc{O}$ satisfying the condition in Lemma \ref{line-bundle-operator:lemma}.

Our main theorem in this subsection is that the algebraic quantum difference equation is well-defined on the weighted subspace $V_{\mbf{n}}\subset V\in\mc{O}$.

\begin{thm}\label{holonomicity-qde:theorem}
Fix the weighted subspace $V_{\mbf{n}}\subset V$ for a representation $V\in\mc{O}$ in the category $\mc{O}$. Then the algebraic quantum difference equation \ref{defn-qde-equation} with the image in $\text{End}(V_{\mbf{n}})((z^{\pm1}))$ is well-defined in the sense that the corresponding difference connection with different $\mc{L}$ commutes:
\begin{align*}
[T_{\mc{L}}^{-1}\mbf{M}_{\mc{L}}^s(z),T_{\mc{L}'}^{-1}\mbf{M}_{\mc{L}'}^s(z)]=0,\qquad\forall\mc{L},\mc{L}'\in\mbb{Z}^I
\end{align*}
\end{thm}

\begin{proof}
This follows from the following computation:
\begin{equation*}
\begin{aligned}
&T_{\mc{L}}^{-1}\mbf{B}_{\mc{L}}^s(z)\mc{L}T_{\mc{L}'}^{-1}\mbf{B}_{\mc{L}'}^s(z)\mc{L}=T_{\mc{L}}^{-1}\tau_{\mc{L}}\prod_{\mbf{m}\in[s,s-\mc{L})}^{\leftarrow}\mbf{B}_{\mbf{m}}(z)\mc{L}T_{\mc{L}'}^{-1}\tau_{\mc{L}'}\prod_{\mbf{m}\in[s,s-\mc{L}')}^{\leftarrow}\mbf{B}_{\mbf{m}}(z)\mc{L}'\\
=&\prod_{\mbf{m}\in[s+\mc{L},s)}^{\leftarrow}\mbf{B}_{\mbf{m}}(z)T_{\mc{L}}^{-1}\mc{L}\prod_{\mbf{m}\in[s+\mc{L}',s)}^{\leftarrow}\mbf{B}_{\mbf{m}}(z)T_{\mc{L}'}^{-1}\mc{L}'\\
=&\prod_{\mbf{m}\in[s+\mc{L},s)}^{\leftarrow}\mbf{B}_{\mbf{m}}(z)T_{\mc{L}}^{-1}\tau_{\mc{L}}\prod_{\mbf{m}\in[s+\mc{L}',s)}^{\leftarrow}\mbf{B}_{\mbf{m}}(z)T_{\mc{L}'}^{-1}\mc{L}'\mc{L}\\
=&\prod_{\mbf{m}\in[s+\mc{L},s)}^{\leftarrow}\mbf{B}_{\mbf{m}}(z)\prod_{\mbf{m}\in[s+\mc{L}+\mc{L}',s+\mc{L}')}^{\leftarrow}\mbf{B}_{\mbf{m}}(z)T_{\mc{L}\mc{L}'}^{-1}\mc{L}\mc{L}'\\
=&T_{\mc{L}\mc{L}'}^{-1}\mbf{B}^s_{\mc{L}\mc{L}'}(z)\mc{L}\mc{L}'.
\end{aligned}
\end{equation*}

On the other side of the computation for $T_{\mc{L}'}^{-1}\mbf{B}_{\mc{L}'}^s(z)\mc{L}'T_{\mc{L}}^{-1}\mbf{B}_{\mc{L}}^s(z)\mc{L}$, we still have the same result as $T_{\mc{L}\mc{L}'}^{-1}\mbf{B}^s_{\mc{L}\mc{L}'}(z)\mc{L}\mc{L}'$. Thus the result has been proved.

\end{proof}

\subsection{Commutativity with the qKZ equations}
\subsubsection{n-point qKZ operator}
qKZ equation was first proposed by \cite{FR92} as a set of $p$-difference equations to express the intertwining operators for the modules of the quantum affine algebras.

The general qKZ equation can be expressed in the following way: Given a positive integer $n$, the $i$-the qKZ operator is defined as:
\begin{align*}
\mc{A}_{i}^s(a_1,\cdots,a_{n+1}):=\mc{R}^s_{i-1,i}(\frac{a_{i-1}}{a_{i}p})^{-1}\cdots\mc{R}^s_{1,i}(\frac{a_{1}}{a_{i}p})^{-1}z_{(i)}\mc{R}_{i,n+1}^{s}(\frac{a_i}{a_{N+1}})\cdots\mc{R}_{i,i+1}^s(\frac{a_i}{a_{i+1}}).
\end{align*}

For example, when $n=1$:
\begin{align*}
\mc{A}_1^s(a_1,a_2)=z_{(1)}\mc{R}^s_{12}(\frac{a_1}{a_2}),\qquad\mc{A}_2^s(a_1,a_2)=\mc{R}^s_{12}(\frac{a_1}{a_2p})^{-1}z_{(2)}.
\end{align*}

The $n$-point qKZ equation is written as:
\begin{align*}
\Psi(a_1,\cdots,a_ip,\cdots,a_{n+1})=\mc{A}^s_i(a_1,\cdots,a_{n+1})\Psi(a_1,\cdots,a_i,\cdots,a_{n+1}).
\end{align*}

One can prove that the qKZ operator commutes with each other:
\begin{thm}
The qKZ equation is a holonomic difference module over the equivariant variable, i.e. the qKZ difference operators commute with each other.
\end{thm}
\begin{proof}
The proof is basically the same as the one in \cite{FR92}. We need to prove the following identity:
\begin{align*}
\mc{A}^s_i(a_1,\cdots,pa_j,\cdots,a_{n+1})\mc{A}^s_j(a_1,\cdots,a_{n+1})=\mc{A}^s_{j}(a_1,\cdots,pa_i,\cdots,a_{n+1})\mc{A}^s_i(a_1,\cdots,a_{n+1}).
\end{align*}

Without loss of generality, we can assume that $i<j$, and in this we write down the left hand side as:
\begin{equation*}
\begin{aligned}
&\mc{R}^s_{i-1,i}(\frac{a_{i-1}}{a_{i}p})^{-1}\cdots\mc{R}^s_{1,i}(\frac{a_{1}}{a_{i}p})^{-1}z_{(i)}\mc{R}_{i,n+1}^{s}(\frac{a_i}{a_{N+1}})\cdots\mc{R}^s_{i,j}(\frac{a_i}{pa_j})\cdots\mc{R}_{i,i+1}^s(\frac{a_i}{a_{i+1}})\times\\
\times&\mc{R}^s_{j-1,j}(\frac{a_{j-1}}{a_{j}p})^{-1}\cdots\mc{R}^s_{i,j}(\frac{a_i}{a_jp})^{-1}\cdots\mc{R}^s_{1,j}(\frac{a_{1}}{a_{j}p})^{-1}z_{(j)}\mc{R}_{j,n+1}^{s}(\frac{a_j}{a_{N+1}})\cdots\mc{R}_{j,j+1}^s(\frac{a_j}{a_{j+1}}).
\end{aligned}
\end{equation*}
Then we move the term $\mc{R}^s_{i,i+1}(\frac{a_i}{a_j})$ to the left hand side of $\mc{R}^s_{i+1,j}(\frac{a_{i+1}}{a_jp})^{-1}\mc{R}^s_{i,j}(\frac{a_{i}}{a_jp})^{-1}$. Using the Yang-Baxter equation:
\begin{align*}
\mc{R}^s_{i,i+1}(\frac{a_i}{a_{i+1}})\mc{R}^s_{i+1,j}(\frac{a_{i+1}}{a_jp})^{-1}\mc{R}^s_{i,j}(\frac{a_{i}}{a_jp})^{-1}=\mc{R}^s_{i,j}(\frac{a_{i}}{a_jp})^{-1}\mc{R}^s_{i+1,j}(\frac{a_{i+1}}{a_jp})^{-1}\mc{R}^s_{i,i+1}(\frac{a_i}{a_{i+1}}).
\end{align*}

\end{proof}

\subsubsection{Commutativity with qKZ operators}
The main theorem of the subsection is that the $n$-point qKZ operator commutes with the algebraic quantum difference operator:
\begin{thm}\label{qKZ-commutes-with-qde:theorem}
For the fixed tensor product of semisimple modules $V\in\mc{O}$, the $qKZ$ operator commutes with the algebraic quantum difference opeartor:
\begin{align*}
[T_{a_i}^{-1}\mc{A}^s_{i}(a_1,\cdots,a_{n+1}),\Delta_s^{n}(\mbf{M}_{\mc{L}}^s(z))]=0.
\end{align*}
\end{thm}
\begin{proof}
We first prove the theorem for the case when $n=1$. It is reduced to prove for $i=1$. It has been known that:
\begin{align*}
\Delta_s(\mbf{M}_{\mc{L}}^s(z))=W_{\mbf{m}_0}(z)W_{\mbf{m}_1}(z)\cdots W_{\mbf{m}_{n-1}}(z)(\mc{L}\otimes\mc{L})T_{\mc{L}}^{-1},\qquad W_{\mbf{m}}(z)=\Delta_{\mbf{m}}(\mbf{B}_{\mbf{m}}(z))(R_{\mbf{m}}^-)
\end{align*}
and using the relation:
\begin{align*}
T_{a_1}^{-1}\mc{A}^s_{1}(a_1,a_2)W_{\mbf{m}}(z)=W_{\mbf{m}}(z)T_{a_1}^{-1}\mc{A}^{s'}_1(a_1,a_2).
\end{align*}

Thus we have that:
\begin{equation*}
\begin{aligned}
&T_{a_1}^{-1}\mc{A}^s_{1}(a_1,a_2)\Delta_s(\mbf{M}_{\mc{L}}^s(z))=T_{a_1}^{-1}\mc{A}^s_{1}(a_1,a_2)W_{\mbf{m}_0}(z)W_{\mbf{m}_1}(z)\cdots W_{\mbf{m}_n}(z)(\mc{L}\otimes\mc{L})T_{\mc{L}}^{-1}\\
=&W_{\mbf{m}_0}(z)W_{\mbf{m}_1}(z)\cdots W_{\mbf{m}_n}(z)T_{a_1}^{-1}\mc{A}^{s+\mc{L}}_{1}(a_1,a_2)(\mc{L}\otimes\mc{L})T_{\mc{L}}^{-1}\\
=&W_{\mbf{m}_0}(z)W_{\mbf{m}_1}(z)\cdots W_{\mbf{m}_n}(z)(\mc{L}\otimes\mc{L})T_{\mc{L}}^{-1}T_{a_1}^{-1}\mc{A}^{s}_{1}(a_1,a_2)\\
=&\Delta_s(\mbf{M}_{\mc{L}}^s(z))T_{a_1}^{-1}\mc{A}^{s}_{1}(a_1,a_2).
\end{aligned}
\end{equation*}

Then we prove the theorem for the case when $i=1$, i.e.
\begin{align*}
[T_{u_1}^{-1}z_{(1)}\mc{R}_{1,n+1}^s(\frac{a_1}{a_{n+1}})\cdots\mc{R}_{1,2}^s(\frac{a_1}{a_2}),\Delta_s^n(\mbf{M}_{\mc{L}}^s(z))]=0.
\end{align*}

Now for the general $n$, we first compute $\Delta_{s}^n(\mbf{M}^s_{\mc{L}}(z))$:
\begin{equation*}
\begin{aligned}
&\Delta_{s}^n(\mbf{M}^s_{\mc{L}}(z))=(\text{Id}^{\otimes n-1}\otimes\Delta_s)\cdots\Delta_s(\mbf{M}^s_{\mc{L}}(z))\\
=&(\text{Id}^{\otimes n-1}\otimes\Delta_s)\cdots(\text{Id}\otimes\Delta_s)(W_{\mbf{m}_0}(z)\cdots W_{\mbf{m}_n}(z)(\mc{L}\otimes\mc{L})T_{\mc{L}}^{-1})\\
=&W_{\mbf{m}_0}^{(n)}(z)\cdots W_{\mbf{m}_n}^{(n)}(z)\Delta_{s}^{n-1}(\mc{L}\otimes\mc{L})T_{\mc{L}}^{-1}.
\end{aligned}
\end{equation*}

Here:
\begin{align*}
W_{\mbf{m}}^{(n)}(z)=\Delta_{\mbf{m}}^n(\mbf{B}_{\mbf{m}}(z))(R_{\mbf{m}}^-)_{1,2}\cdots(R_{\mbf{m}}^-)_{1,n+1}.
\end{align*}

When $n=1$, it has been known that:
\begin{align*}
\Delta_{\mbf{m}}(\mbf{B}_{\mbf{m}}(z))=\mbf{J}_{\mbf{m}}^+(z)(\mbf{B}_{\mbf{m}}(zq^{\kappa_{(2)}})\otimes\mbf{B}_{\mbf{m}}(zq^{-\kappa_{(1)}}))\mbf{J}_{\mbf{m}}^-(z)
\end{align*}
and by the definition of $\Delta_{\mbf{m}}^n$ we have that:
\begin{equation*}
\begin{aligned}
\Delta_{\mbf{m}}^n(\mbf{B}_{\mbf{m}}(z))=&\prod_{k=1}^{\substack{\rightarrow\\n-1}}(\prod_{i=1}^{n-k}(\Delta_{\mbf{m}}\otimes\text{Id}^{n-i}))(\mbf{J}_{\mbf{m}}^{+}(zq^{\kappa_{(n)}+\cdots+\kappa_{(n-k+1)}}))(\otimes_{l=1}^{n+1}\mbf{B}_{\mbf{m}}(zq^{\sum_{i\neq l}(-1)^{\theta(i-l)}\kappa_{i}}))\times\\
&\times\prod_{k=1}^{\substack{\leftarrow\\n-1}}(\prod_{i=1}^{n-k}(\Delta_{\mbf{m}}\otimes\text{Id}^{n-i}))(\mbf{J}_{\mbf{m}}^{-}(zq^{\kappa_{(n)}+\cdots+\kappa_{(n-k+1)}}))
\end{aligned}
\end{equation*}
and here $\theta(x)$ is the step function:
\begin{align*}
\theta(x)=
\begin{cases}
0&x\geq0\\
1&x<0
\end{cases}.
\end{align*}

For example, when $n=3$, we have that:
\begin{equation*}
\begin{aligned}
&\Delta_{\mbf{m}}^3(\mbf{B}_{\mbf{m}}(z))=\mbf{J}_{\mbf{m}}^+(z)^{12,3,4}\mbf{J}_{\mbf{m}}^+(zq^{\kappa_{(4)}})^{12,3}\mbf{J}_{\mbf{m}}^+(zq^{\kappa_{(3)}+\kappa_{(4)}})^{12}\times\\
&\times(\mbf{B}_{\mbf{m}}(zq^{\kappa_{(2)}+\kappa_{(3)}+\kappa_{(4)}})\otimes\mbf{B}_{\mbf{m}}(zq^{-\kappa_{(1)}+\kappa_{(3)}+\kappa_{(4)}})\otimes\mbf{B}_{\mbf{m}}(zq^{-\kappa_{(1)}-\kappa_{(2)}+\kappa_{(4)}})\otimes\mbf{B}_{\mbf{m}}(zq^{-\kappa_{(1)}-\kappa_{(2)}-\kappa_{(3)}}))\\
&\times\mbf{J}_{\mbf{m}}^-(zq^{\kappa_{(3)}+\kappa_{(4)}})^{12}\mbf{J}_{\mbf{m}}^-(zq^{\kappa_{(4)}})^{12,3}\mbf{J}_{\mbf{m}}^-(z)^{12,3,4}.
\end{aligned}
\end{equation*}

The following lemma is useful in the computation.
\begin{lem}\label{higher-coproduct-fusion-operator-ABBR-equation:label}
The operator $J(z):=\prod_{k=1}^{\substack{\leftarrow\\n-1}}(\prod_{i=1}^{n-k}(\Delta_{\mbf{m}}\otimes\text{Id}^{n-i}))(\mbf{J}_{\mbf{m}}^{-}(zq^{\kappa_{(n)}+\cdots+\kappa_{(n-k+1)}}))$ satisfies the following equation:
\begin{align*}
J(z)T_{a_1}z_{(1)}^{-1}(R_{\mbf{m}}^-)_{1,n+1}^{-1}\cdots(R_{\mbf{m}}^-)_{1,2}^{-1}=T_{a_1}z_{(1)}^{-1}q^{\Omega_{1,n+1}+\cdots\Omega_{1,2}}J(z)
\end{align*}
and similarly, the operator $J(z):=\prod_{k=1}^{\substack{\rightarrow\\n-1}}(\prod_{i=1}^{n-k}(\Delta_{\mbf{m}}\otimes\text{Id}^{n-i}))(\mbf{J}_{\mbf{m}}^{+}(zq^{\kappa_{(n)}+\cdots+\kappa_{(n-k+1)}}))$ satisfies the following equation:
\begin{align*}
(R_{\mbf{m}}^+)_{1,2}^{-1}\cdots(R_{\mbf{m}}^+)_{1,n+1}^{-1}T_{a_1}z_{(1)}^{-1}J(z)=J(z)q^{\Omega_{1,2}+\cdots+\Omega_{1,n+1}}T_{a_1}z_{(1)}^{-1}.
\end{align*}
\end{lem}
\begin{proof}
If we denote $J(z):=\prod_{k=1}^{\substack{\leftarrow\\n-1}}(\prod_{i=1}^{n-k}(\Delta_{\mbf{m}}\otimes\text{Id}^{n-i}))(J_{\mbf{m}}^{-}(zq^{\kappa_{(n)}+\cdots+\kappa_{(n-k+1)}}))$. By computation we have that:
\begin{equation*}
\begin{aligned}
&\prod_{k=1}^{\substack{\leftarrow\\n-1}}(\prod_{i=1}^{n-k}(\Delta_{\mbf{m}}\otimes\text{Id}^{n-i}))(J_{\mbf{m}}^{-}(zq^{\kappa_{(n)}+\cdots+\kappa_{(n-k+1)}}))z_{(n+1)}(R_{\mbf{m}}^-)^{-1}_{1,n+1}\cdots(R_{\mbf{m}}^-)^{-1}_{1,2}\\
=&\prod_{k=1}^{\substack{\leftarrow\\n-2}}(\prod_{i=1}^{n-k}(\Delta_{\mbf{m}}\otimes\text{Id}^{n-i}))(J_{\mbf{m}}^{-}(zq^{\kappa_{(n)}+\cdots+\kappa_{(n-k+1)}}))\prod_{i=1}^{n-1}(\Delta_{\mbf{m}}\otimes\text{Id}^{n-i})J_{\mbf{m}}^-(z)z_{(n+1)}(R_{\mbf{m}}^-)^{-1}_{1,n+1}\cdots(R_{\mbf{m}}^-)^{-1}_{1,2}\\
=&\prod_{k=1}^{\substack{\leftarrow\\n-2}}(\prod_{i=1}^{n-k}(\Delta_{\mbf{m}}\otimes\text{Id}^{n-i}))(J_{\mbf{m}}^{-}(zq^{\kappa_{(n)}+\cdots+\kappa_{(n-k+1)}}))\prod_{i=1}^{n-1}(\Delta_{\mbf{m}}\otimes\text{Id}^{n-i})J_{\mbf{m}}^-(z)z_{(1)}^{-1}\cdots z_{(n)}^{-1}(R_{\mbf{m}}^-)^{-1}_{1,n+1}\cdots(R_{\mbf{m}}^-)^{-1}_{1,2}\\
=&\prod_{k=1}^{\substack{\leftarrow\\n-2}}(\prod_{i=1}^{n-k}(\Delta_{\mbf{m}}\otimes\text{Id}^{n-i}))(J_{\mbf{m}}^{-}(zq^{\kappa_{(n)}+\cdots+\kappa_{(n-k+1)}}))z_{(1)}^{-1}\cdots z_{(n)}^{-1}q^{\Omega_{1,n+1}+\cdots+\Omega_{n,n+1}}\prod_{i=1}^{n-1}(\Delta_{\mbf{m}}\otimes\text{Id}^{n-i})J_w^-(z)\\
=&\prod_{k=1}^{\substack{\leftarrow\\n-2}}(\prod_{i=1}^{n-k}(\Delta_{\mbf{m}}\otimes\text{Id}^{n-i}))(J_{\mbf{m}}^{+}(zq^{\kappa_{(n)}+\cdots+\kappa_{(n-k+1)}}))z_{(n+1)}q^{\Omega_{1,n+1}+\cdots+\Omega_{n,n+1}}\prod_{i=1}^{n-1}(\Delta_{\mbf{m}}\otimes\text{Id}^{n-i})J_{\mbf{m}}^-(z)\\
=&z_{(n+1)}q^{\Omega_{1,n+1}+\cdots+\Omega_{n,n+1}}\prod_{k=1}^{\substack{\leftarrow\\n-2}}(\prod_{i=1}^{n-k}(\Delta_{\mbf{m}}\otimes\text{Id}^{n-i}))(J_{\mbf{m}}^{-}(zq^{\kappa_{(n)}+\cdots+\kappa_{(n-k+1)}}))\prod_{i=1}^{n-1}(\Delta_{\mbf{m}}\otimes\text{Id}^{n-i})J_{\mbf{m}}^-(z).
\end{aligned}
\end{equation*}

Thus we can see that $J(z)$ satisfies the following equation:
\begin{align*}
J(z)z_{(n+1)}(R_{\mbf{m}}^{-})^{-1}_{1,n+1}\cdots(R_{\mbf{m}}^-)^{-1}_{n,n+1}=z_{(n+1)}q^{\Omega_{1,n+1}+\cdots+\Omega_{n,n+1}}J(z)
\end{align*}
where we can show that it is also the solution to the equation:
\begin{align*}
J(z)z_{(n+1)}(R_{\mbf{m}}^-)^{-1}_{1,n+1}\cdots(R_{\mbf{m}}^-)^{-1}_{n,n+1}=z_{(n+1)}q^{\Omega_{1,n+1}+\cdots+\Omega_{n,n+1}}J(z).
\end{align*}

Next we show that $J(z)$ also satisfies the following equation:
\begin{align*}
J(z)z_{(1)}^{-1}(R_{\mbf{m}}^-)^{-1}_{1,n+1}\cdots(R_{\mbf{m}}^-)^{-1}_{1,2}=z_{(1)}^{-1}q^{\Omega_{1,n+1}+\cdots+\Omega_{1,2}}J(z).
\end{align*}

We define two operators:
\begin{equation*}
\begin{aligned}
&A_{R}(X):=q^{-\Omega_{1,n+1}-\cdots-\Omega_{n,n+1}}z_{(n+1)}^{-1}Xz_{(n+1)}(R_{\mbf{m}}^-)^{-1}_{1,n+1}\cdots(R_{\mbf{m}}^-)^{-1}_{n,n+1}\\
&A_{L}(X):=q^{-\Omega_{1,n+1}-\cdots-\Omega_{1,2}}z_{(1)}Xz_{(1)}^{-1}(R_{\mbf{m}}^-)^{-1}_{1,n+1}\cdots(R_{\mbf{m}}^-)^{-1}_{1,2}
\end{aligned}
\end{equation*}
and by the Yang-Baxter equation, $[A_{L},A_{R}]=0$. Thus if $A_{R}(X)=X$, we have that $A_{L}(X)$ is also a solution. Similarly, one can see that the solution to the equation $A_{R}(X)=X$ is uniquely determined by the degree zero part of the $(n+1)$-th component, and we denote the corresponding component as $X_0$. It remains to prove that $A_L(X)_0=X_0$. For $X_0$ we have that:
\begin{align*}
X_0=J_{\mbf{m}}^-(zq^{\kappa_{(n)}+\cdots+\kappa_{(3)}})^{12}
\end{align*}
and the proof here is totally similar to the one given in Proposition 10 in \cite{OS22}. Thus the proof is finished.
\end{proof}

Let us denote the operator:
\begin{align*}
\mc{K}_{\mbf{m},n+1}^-=T_{a_1}z_{(1)}^{-1}(R_{\mbf{m}}^-)^{-1}_{1,n+1}\cdots(R_{\mbf{m}}^-)^{-1}_{1,2},\qquad\mc{K}_{\mbf{m},n+1}^{+}=(R_{\mbf{m}}^+)^{-1}_{1,2}\cdots(R_{\mbf{m}}^+)^{-1}_{1,n+1}T_{a_1}z_{(1)}^{-1}.
\end{align*}

\begin{prop}
\begin{equation*}
\mc{K}_{\mbf{m},n+1}^{+}\Delta_{\mbf{m}}^n(\mbf{B}_{\mbf{m}}(z))=\Delta_{\mbf{m}}^n(\mbf{B}_{\mbf{m}}(z))\mc{K}_{\mbf{m},n+1}^-.
\end{equation*}
\end{prop}
\begin{proof}
This can be seen from the direct computations:
\begin{equation*}
\begin{aligned}
&\mc{K}_{\mbf{m},n+1}^{+}\Delta_{\mbf{m}}^n(\mbf{B}_{\mbf{m}}(z))\\
=&(R_{\mbf{m}}^+)^{-1}_{1,2}\cdots(R_{\mbf{m}}^+)^{-1}_{1,n+1}T_{a_1}z_{(1)}^{-1}\prod_{k=1}^{\substack{\rightarrow\\n-1}}(\prod_{i=1}^{n-k}(\Delta_{\mbf{m}}\otimes\text{Id}^{n-i}))(\mbf{J}_{\mbf{m}}^{+}(zq^{\kappa_{(n)}+\cdots+\kappa_{(n-k+1)}}))\times\\
&\times(\otimes_{l=1}^{n+1}\mbf{B}_{\mbf{m}}(zq^{\sum_{i\neq l}(-1)^{\theta(i-l)}\kappa_{i}}))\prod_{k=1}^{\substack{\leftarrow\\n-1}}(\prod_{i=1}^{n-k}(\Delta_{\mbf{m}}\otimes\text{Id}^{n-i}))(\mbf{J}_{\mbf{m}}^{+}(zq^{\kappa_{(n)}+\cdots+\kappa_{(n-k+1)}}))\\
=&\prod_{k=1}^{\substack{\rightarrow\\n-1}}(\prod_{i=1}^{n-k}(\Delta_{\mbf{m}}\otimes\text{Id}^{n-i}))(\mbf{J}_{\mbf{m}}^{+}(zq^{\kappa_{(n)}+\cdots+\kappa_{(n-k+1)}}))(\otimes_{l=1}^{n+1}\mbf{B}_{\mbf{m}}(zq^{\sum_{i\neq l}(-1)^{\theta(i-l)}\kappa_{i}}))\\
&\times T_{a_1}z_{(1)}^{-1}q^{\Omega_{1,n+1}+\cdots+\Omega_{1,2}}\prod_{k=1}^{\substack{\leftarrow\\n-1}}(\prod_{i=1}^{n-k}(\Delta_{\mbf{m}}\otimes\text{Id}^{n-i}))(\mbf{J}_{\mbf{m}}^{+}(zq^{\kappa_{(n)}+\cdots+\kappa_{(n-k+1)}}))\\
=&\prod_{k=1}^{\substack{\rightarrow\\n-1}}(\prod_{i=1}^{n-k}(\Delta_{\mbf{m}}\otimes\text{Id}^{n-i}))(\mbf{J}_{\mbf{m}}^{+}(zq^{\kappa_{(n)}+\cdots+\kappa_{(n-k+1)}}))(\otimes_{l=1}^{n+1}\mbf{B}_{\mbf{m}}(zq^{\sum_{i\neq l}(-1)^{\theta(i-l)}\kappa_{i}}))\\
&\times\prod_{k=1}^{\substack{\leftarrow\\n-1}}(\prod_{i=1}^{n-k}(\Delta_{\mbf{m}}\otimes\text{Id}^{n-i}))(\mbf{J}_{\mbf{m}}^{+}(zq^{\kappa_{(n)}+\cdots+\kappa_{(n-k+1)}}))T_{a_1}z_{(1)}^{-1}(R_{\mbf{m}}^-)^{-1}_{1,n+1}\cdots(R_{\mbf{m}}^-)^{-1}_{1,2}\\
=&\Delta_{\mbf{m}}^n(\mbf{B}_{\mbf{m}}(z))\mc{K}_{\mbf{m},n+1}^-
\end{aligned}
\end{equation*}
the second and third equalities come from Lemma \ref{higher-coproduct-fusion-operator-ABBR-equation:label}.
Thus the proof is finished.

\end{proof}

Given $n\geq1$, we define the operator:
\begin{align*}
W_{\mbf{m}}^{(n)}(z):=\Delta_{\mbf{m}}^{n-1}(\Delta_{\mbf{m}}(\mbf{B}_{\mbf{m}}(z))R_{\mbf{m}}^-).
\end{align*}

Using the above proposition, we can have the following commutation relation between the qKZ operator and the dynamical monodromy operator:

\begin{prop}
\begin{equation*}
\mc{A}^{s,qKZ,(n)}_{1}W_{s}^{(n)}(z)=W_{s}^{(n)}(z)z_{(1)}T_{a_1}^{-1}\Delta_{s}^{n-1}(\mc{R}^{s'}).
\end{equation*}

Here $\Delta_{s}^{n-1}(\mc{R}^{s'})$ means the image of $\Delta_{s}^{n-1}(\mc{R}^{s'})$ under the evaluation representation $V_1(a_1)\otimes\cdots\otimes V_{n+1}(a_{n+1})$ with $V_1,\cdots,V_{n+1}\in\mc{O}$.
\end{prop}
\begin{proof}
\begin{equation*}
\begin{aligned}
&\mc{A}^{s,qKZ,(n)}_{1}W_{s}^{(n)}(z)=T_{a_1}^{-1}z_{(1)}\mc{R}_{1,n+1}^s(\frac{a_1}{a_{n+1}})\cdots\mc{R}_{1,2}^s(\frac{a_1}{a_2})\Delta_{s}^n(\mbf{B}_{s}(z))(R_{s}^-)_{n,n+1}\cdots(R_{s}^-)_{1,n}\\
=&T_{a_1}^{-1}z_{(1)}(\Delta_{s}^n)^{op}(\mbf{B}_{s}(z))\mc{R}_{1,n+1}^s(\frac{a_1}{a_{n+1}})\cdots\mc{R}_{1,2}^s(\frac{a_1}{a_2})(R_{s}^-)_{n,n+1}\cdots(R_{s}^-)_{1,n}\\
=&T_{a_1}^{-1}z_{(1)}((R_{s}^{+})^{-1}_{1,2}\cdots(R_{s}^{+})^{-1}_{1,n+1})^{-1}(R_{s}^{+})^{-1}_{1,2}\cdots(R_{s}^{+})^{-1}_{1,n+1}(\Delta_{s}^n)^{op}(\mbf{B}_{s}(z))\times\\
&\times\mc{R}_{1,n+1}^s(\frac{a_1}{a_{n+1}})\cdots\mc{R}_{1,2}^s(\frac{a_1}{a_2})(R_{s}^-)_{n,n+1}\cdots(R_{s}^-)_{1,n}\\
=&T_{a_1}^{-1}z_{(1)}((R_{s}^{+})^{-1}_{1,2}\cdots(R_{s}^{+})^{-1}_{1,n+1})^{-1}(\Delta_{s}^n)(\mbf{B}_{s}(z))(R_{s}^{+})^{-1}_{1,2}\cdots(R_{s}^{+})^{-1}_{1,n+1}\times\\
&\times\mc{R}_{1,n+1}^s(\frac{a_1}{a_{n+1}})\cdots\mc{R}_{1,2}^s(\frac{a_1}{a_2})(R_{s}^-)_{1,n+1}\cdots(R_{s}^-)_{1,n}\\
=&(\Delta_{s}^n)(\mbf{B}_{s}(z))((R_{s}^{-})^{-1}_{1,n+1}\cdots(R_{s}^{-})^{-1}_{1,2})^{-1}z_{(1)}T_{a_1}^{-1}\Delta_{s}^{n-1}(\mc{R}^{s'})
\end{aligned}
\end{equation*}
and here the last equality comes from Lemma \ref{higher-coproduct-fusion-operator-ABBR-equation:label} and the formula \eqref{universal-R-matrix-of-different-slopes}.

\end{proof}
Now combining the above result, we have that:
\begin{equation*}
\begin{aligned}
&\mc{A}^{s,qKZ,(n)}_1\Delta_{s}^{n}(\mbf{M}_{\mc{L}}^s(z))=\mc{A}^{s,qKZ,(n)}W_{\mbf{m}_0}^{(n)}(z)\cdots W_{\mbf{m}_l}^{(n)}(z)\Delta_{s}^{n-1}(\mc{L}\otimes\mc{L})T_{\mc{L}}^{-1}\\
=&W_{\mbf{m}_0}^{(n)}(z)\cdots W_{\mbf{m}_l}^{(n)}(z)z_{(1)}T_{a_1}^{-1}\Delta_{s}^{n-1}(\mc{R}^{s+\mc{L}}(\mc{L}\otimes\mc{L}))T_{\mc{L}}^{-1}\\
=&W_{\mbf{m}_0}^{(n)}(z)\cdots W_{\mbf{m}_l}^{(n)}(z)z_{(1)}T_{a_1}^{-1}\Delta_{s}^{n-1}((\mc{L}\otimes\mc{L})\mc{R}^{s})T_{\mc{L}}^{-1}\\
=&\Delta_{s}^{n}(\mbf{M}_{\mc{L}}^s(z))\mc{A}^{s,qKZ,(n)}_1.
\end{aligned}
\end{equation*}

Thus we have proved the theorem for $i=1$.

For the general $i$, note that the theorem can be written in the following way:
\begin{equation*}
\begin{aligned}
&\mc{R}^{s}_{i-1,i}(\frac{a_{i-1}}{a_{i}})^{-1}\cdots\mc{R}^s_{1,i}(\frac{a_1}{a_i})^{-1}T_{a_i}^{-1}z_{(i)}\mc{R}^s_{i,n+1}(\frac{a_{i}}{a_{N+1}})\cdots\mc{R}^s_{i,i+1}(\frac{a_i}{a_{i+1}})\Delta_{s}^n(\mc{B}^s_{\mc{L}})\\
=&\Delta_{s}^n(\mc{B}^s_{\mc{L}})\mc{R}^{s}_{i-1,i}(\frac{a_{i-1}}{a_{i}})^{-1}\cdots\mc{R}^s_{1,i}(\frac{a_1}{a_i})^{-1}T_{a_i}^{-1}z_{(i)}\mc{R}^s_{i,n+1}(\frac{a_{i}}{a_{N+1}})\cdots\mc{R}^s_{i,i+1}(\frac{a_i}{a_{i+1}})
\end{aligned}
\end{equation*}
and this is equivalent to the following equation:
\begin{equation*}
\begin{aligned}
&T_{a_i}^{-1}z_{(i)}\mc{R}^s_{i,n+1}(\frac{a_{i}}{a_{N+1}})\cdots\mc{R}^s_{i,i+1}(\frac{a_i}{a_{i+1}})\Delta_{s}^n(\mc{B}^s_{\mc{L}})=\mc{R}_{1,i}^{s}(\frac{a_1}{a_i})\cdots\mc{R}_{i-1,i}^s(\frac{a_{i-1}}{a_i})\times\\
&\times\Delta_{s}^n(\mc{B}^s_{\mc{L}})\mc{R}^{s}_{i-1,i}(\frac{a_{i-1}}{a_{i}})^{-1}\cdots\mc{R}^s_{1,i}(\frac{a_1}{a_i})^{-1}T_{a_i}^{-1}z_{(i)}\mc{R}^s_{i,n+1}(\frac{a_{i}}{a_{N+1}})\cdots\mc{R}^s_{i,i+1}(\frac{a_i}{a_{i+1}})\\
=&(\Delta_{s}^{op})^{i-1}\Delta_{s}^{n-i+1}(\mc{B}^s_{\mc{L}})T_{a_i}^{-1}z_{(i)}\mc{R}^s_{i,n+1}(\frac{a_{i}}{a_{N+1}})\cdots\mc{R}^s_{i,i+1}(\frac{a_i}{a_{i+1}}).
\end{aligned}
\end{equation*}

On the other hand the formula can be rewritten as:
\begin{equation*}
\begin{aligned}
&T_{a_i}^{-1}z_{(i)}\mc{R}^s_{i,n+1}(\frac{a_{i}}{a_{N+1}})\cdots\mc{R}^s_{i,i+1}(\frac{a_i}{a_{i+1}})\Delta_{s}^n(\mbf{M}^s_{\mc{L}}(z))\\
=&T_{a_i}^{-1}z_{(i)}\mc{R}^s_{i,n+1}(\frac{a_{i}}{a_{N+1}})\cdots\mc{R}^s_{i,i+1}(\frac{a_i}{a_{i+1}})\mc{R}^{s}_{i,i-1}(\frac{a_i}{a_{i-1}})\cdots\mc{R}^s_{i,1}(\frac{a_i}{a_1})\times\\
&\times(\Delta_s^{op})^{i-1}\Delta_{s}^{n-i+1}(\mbf{M}_{\mc{L}}^s(z))\mc{R}^s_{i,1}(\frac{a_i}{a_1})^{-1}\cdots\mc{R}^s_{i,i-1}(\frac{a_i}{a_{i-1}})^{-1}.
\end{aligned}
\end{equation*}

Thus the identity has been reduced to the following form:
\begin{equation*}
\begin{aligned}
&T_{a_i}^{-1}z_{(i)}\mc{R}^s_{i,n+1}(\frac{a_{i}}{a_{N+1}})\cdots\mc{R}^s_{i,i+1}(\frac{a_i}{a_{i+1}})\mc{R}^{s}_{i,i-1}(\frac{a_i}{a_{i-1}})\cdots\mc{R}^s_{i,1}(\frac{a_i}{a_1})(\Delta_s^{op})^{i-1}\Delta_{s}^{n-i+1}(\mbf{M}_{\mc{L}}^s(z))\\
=&(\Delta_{s}^{op})^{i-1}\Delta_{s}^{n-i+1}(\mbf{M}^s_{\mc{L}}(z))T_{a_i}^{-1}z_{(i)}\mc{R}^s_{i,n+1}(\frac{a_{i}}{a_{N+1}})\cdots\mc{R}^s_{i,i+1}(\frac{a_i}{a_{i+1}})\mc{R}^{s}_{i,i-1}(\frac{a_i}{a_{i-1}})\cdots\mc{R}^s_{i,1}(\frac{a_i}{a_1}).
\end{aligned}
\end{equation*}

This identity is equivalent to commuting the qKZ equation of the first variable with the algebraic quantum difference operator via the following braiding exchange:
\begin{equation*}
\begin{aligned}
&V_1\otimes\cdots V_{i-1}\otimes V_i\otimes\cdots\otimes V_{n+1}\\
\rightarrow&V_1\otimes\cdots\otimes V_{i-2}\otimes V_{i}\otimes V_{i-1}\otimes\cdots\otimes V_{n+1}\\
\rightarrow& V_1\otimes\cdots\otimes V_{i}\otimes V_{i-2}\otimes V_{i-1}\otimes\cdots\otimes V_{n+1}\\
\rightarrow&\cdots\cdots\cdots\cdots\\
\rightarrow&V_i\otimes V_1\otimes\cdots V_{i-1}\otimes V_{i+1}\otimes\cdots\otimes V_{n+1}
\end{aligned}
\end{equation*}
and the action of the quantum difference operator $\mbf{M}^s_{\mc{L}}(z)$ on $V_i\otimes V_1\otimes\cdots\otimes V_{i-1}\otimes V_{i+1}\otimes\cdots\otimes V_{n+1}$ coincides with $(\Delta_{s}^{op})^{i-1}\Delta_{s}^{n-i+1}(\mbf{M}^s_{\mc{L}}(z))$. Thus we just need to repeat the above proof for the case $i=1$. Now the proof is finished.
\end{proof}

\section{\textbf{Examples}}\label{sec:_textbf_examples}
In this section we illustrate some examples 
\subsection{Quantum loop group of finite type quivers}
For the case of the quantum loop group of symmetrisable Cartan matrices. If the Cartan matrix is of the finite type, the example was shown in \ref{ssub:quantum_loop_group_associated_to_finite_type_symmetrisable_cartan_matrix}. In this case, the slope subalgebra $\mc{B}_{\mbf{m}}$ can be thought of as the quantum root subalgebras in the quantum loop group $U_q(L\mf{g})$.

We make the following definition. Fix the quantum loop group $U_{q}(L\mf{g})$ with the fixed finite type symmetrisable Cartan matrix $C$. Let $\mf{g}$ be the corresponding Lie algebra of the Cartan matrix type $C$. We denote its Chevalley generators by $e_1,\cdots,e_n,h_1,\cdots,h_n,f_1,\cdots,f_n$. We also denote the corresponding lattice by $\mbb{Z}^n$ such that each triples of generators $e_i,f_i,h_i$ corresponds to the unit vector $\mbf{e}_i\in\mbb{Z}^n$. Thus each $\mf{sl}_2$-triple in $\mf{g}$ corresponds to a point $\bm{\alpha}$ in the lattice $\mbb{Z}^n$.

Given a vector $\mbf{m}\in\mbb{R}^n$, we define the Lie subalgebra $\mf{g}_{\mbf{m}}$ of $\mf{g}$ as the Lie algebra generated by the $\mf{sl}_2$-triple $e_{\bm{\alpha}},h_{\bm{\alpha}},f_{\bm{\alpha}}$ such that $\mbf{m}\cdot\bm{\alpha}\in\mbb{Z}$.

It is obvious that the Lie algebra $\mf{g}_{\mbf{m}}$ is also semisimple. It also admits a natural Hopf algebra deformation to a quantum group, and we denote it by $U_q(\mf{g}_{\mbf{m}})$.

We have the following conjectural description for the slope subalgebra $\mc{B}_{\mbf{m}}$ of $U_{q}(L\mf{g})$.
\begin{conjecture}\label{refined-structure-for-quantum-affine-algebra:conjecture}
The slope subalgebra $\mc{B}_{\mbf{m}}$ for the quantum loop group $U_{q}(L\mf{g})$ of finite type is isomorphic to the quantum group $U_{q}(\mf{g}_{\mbf{m}})$ for some corresponding semisimple Lie algebra $\mf{g}_{\mbf{m}}$ which is a Lie subalgebra of $\mf{g}$ determined by $\mbf{m}$.
\end{conjecture}

For the case when the Cartan matrix $C$ is of the finite $A$ type. This has been proved in \cite{N15} by using the same argument from the case of the quantum toroidal algebras. For the rest of the case, we will prove the conjecture in a forthcoming paper in the near future.

However, for the slope factorisation for the positive/negative half of the quantum loop group $U_{q}^{\pm}(L\mf{g})$. We have the following result:
\begin{thm}\label{generic-factorization-for-quantum-loop-group:theorem}
For the quantum loop group with finite type symmetrisable Cartan matrix, there is a cutty-corner curve $\mu$ such that for the factorisation:
\begin{align*}
U_{q}^{+}(L\mf{g})\cong\bigotimes^{\rightarrow}_{t\in\mbb{R}}\mc{B}_{\mu(t)}^+
\end{align*}
such that each slope subalgebra $\mc{B}_{\mu(t)}^+$ on this curve is either trivial or of the $U_{q}(\mf{sl}_2)^+$ type.
\end{thm}
\begin{proof}

The proof can be given by the following steps:

\textbf{Step 1: }For the generic $\mbf{m}$, it is easy to see that generically for the nontrivial slope subalgebra $\mc{B}_{\mbf{m}}$, $\mbf{m}$ can be chosen generically satisfying the following property: There exists $\mbf{n}\in\mbb{N}^I$ such that $\mbf{m}\cdot\mbf{n}\in\mbb{Z}$ and $\mbf{m}\cdot\mbf{n}'\notin\mbb{Z}$ for $\mbf{n}'\notin\mbb{Z}\mbf{n}$. Thus we can choose the cutty-corner curve $\mu:\mbb{R}\rightarrow\mbb{R}^I$ such that for the slope factorisation:
\begin{align*}
U_{q}^+(L\mf{g})\cong\bigotimes^{\rightarrow}_{t\in\mbb{R}}\mc{B}_{\mu(t)}^+
\end{align*}
the slope subalgebra $\mc{B}_{\mu(t)}$ is of the above generic type.

\textbf{Step 2: }
The quantum loop group $U_{q}(L\mf{g})$ has the Drinfeld-Jimbo realisation. One can also construct the basis for $U_{q}^+(L\mf{g})$ using the Cartan-Weyl basis, and for the details one can also refer to \cite{NT24} for details. The Cartan-Weyl basis corresponds to the affine roots for $U_{q}^+(L\mf{g})$, whcih is is indexed by $\Delta^+\times\mbb{Z}$, which can be identified as the subset in $\mbb{N}^I\times\mbb{Z}$ via identifying each simple root $\alpha_i\in\Delta$ as the $i$-th unit vector $\mbf{e}_i\in\mbb{N}^I$. By the construction in \cite{KT91}, each point in $\Delta^+\times\mbb{Z}\subset\mbb{N}^I\times\mbb{Z}$ will correspond to a $U_{q}^+(\mf{sl}_2)$.

For a generic type slope subalgebra $\mc{B}_{\mbf{m}}^+$ with fixed $\mbf{n}$ such that $\mbf{m}\cdot\mbf{n}\in\mbb{Z}$, i.e.
\begin{align*}
\mc{B}_{\mbf{m}}^+=\bigoplus_{\ell\in\mbb{N}}\mc{B}_{\mbf{m},\ell\mbf{n}}^+.
\end{align*}

Now using the fact that if $\alpha\in\Delta$, then $k\alpha\notin\Delta$ for $k\neq-1,1$. This implies that generically $\mc{B}_{\mbf{m}}^+$ is isomorphic to $U_q^{+}(\mf{sl}_2)$.

\end{proof}

\subsubsection{Dynamical monodromy operator for the $U_{q}(L\mf{sl}_2)$}\label{ssub:dynamical_monodromy_operator_for_the_u_q}
At the very beginning, we compute the dynamical monodromy operator for the quantum loop group $U_q(L\mf{sl}_2)$. In this case the slope subalgebra is nonempty if and only if $m\in\mbb{Z}$, and each slope subalgebra $\mc{B}_{m}$ is isomorphic to $U_{q}(\mf{sl}_2)$ with the generators given by $e_{m}$, $f_{-m}$ and $h_0^{\pm}$. It has the reduced universal $R$-matrix given by:
\begin{align*}
R_{m}^+=\sum_{n=0}^{\infty}\frac{1}{[n]_q!}e_{m}^n\otimes f_{-m}^n.
\end{align*}

Now recall the ABRR equation
\begin{align*}
J_{m}^+(z)z_{(1)}^{-1}q^{\Omega}R_{m}^+=z_{(1)}^{-1}q^{\Omega}J_m^+(z)
\end{align*}
without loss of generality, we can assume that the fusion operator $J_m^+(z)$ can be written as:
\begin{align*}
J_m^+(z)=\sum_{n=0}^{\infty}J_{m,n}^+(z)e_{m}^n\otimes f_{-m}^n.
\end{align*}

This computation has been done in \cite{OS22}, 
and we have that:
\begin{align*}
(1\otimes S_{m})J_m^+(z)=\sum_{k=0}^{\infty}\frac{(-1)^k}{[k]_q!\prod_{i=1}^k(1-z^{-1}(h_0^{+}\otimes h_0^-)q^{2i})}e_m^k\otimes f_{-m}^k.
\end{align*}

Thus for now we have that the dynamical monodromy operator is written as:
\begin{align}\label{sl2-monodromy-operator}
\mbf{B}_m(z)=\sum_{k=0}^{\infty}\mbf{m}[\frac{(-1)^k}{[k]_q!\prod_{i=1}^k(1-z^{-1}p^m(h_0^{+}\otimes h_0^+)q^{2i})}e_m^k\otimes f_{-m}^k].
\end{align}

\subsubsection{General case}

In this case, we will prove the following theorem:
\begin{thm}
For every quantum difference operator $\mbf{M}_{\mc{L}}^s(z):=\mbf{B}^s_{\mc{L}}(z)\mc{L}$, if we fix a weighted subspace $V_{\mbf{n}}$ in a semisimple module $V\in\mc{O}$ there is a generic path from $s$ to $s-\mc{L}$ such that each dynamical monodromy operator $\mbf{B}_{\mbf{m}}(z)$ is written as:
\begin{align*}
\mbf{B}_{\mbf{m}}(z)=\sum_{k=0}^{\infty}\mbf{m}[\frac{(-1)^k\gamma_{\mbf{m}}^k}{[k]_q!\prod_{i=1}^k(1-z^{-l_{\mbf{m}}}p^{l_{\mbf{m}}\cdot\mbf{m}}(h_{\mbf{m}}^+\otimes h_{\mbf{m}}^-)q^{2il_{\mbf{m}}})}e_{\mbf{m}}^k\otimes f_{-\mbf{m}}^{k}]
\end{align*}
where $\gamma_{\mbf{m}}$ is a constant given by $\langle e_{\mbf{m}},f_{-\mbf{m}}\rangle=\gamma_{\mbf{m}}^{-1}$.
\end{thm}
\begin{proof}
By Theorem \ref{generic-factorization-for-quantum-loop-group:theorem}, one can choose a generic path from $s$ to $s-\mc{L}$ in the definition of the quantum difference operator $\mbf{M}_{\mc{L}}^s(z)$ in \eqref{defn-qde-equation}. In the expression:
\begin{align*}
\mbf{M}_{\mc{L}}^s(z)=\mc{L}\prod^{\rightarrow}_{s\in[s,s-\mc{L})}\mbf{B}_{\mbf{m}}(z)
\end{align*}
each dynamical monodromy operator $\mbf{B}_{\mbf{m}}(z)$ comes from the generic slope subalgebra $\mc{B}_{\mbf{m}}$, which is isomorphic to $U_{q}(\mf{sl}_2)$. Now the computation is reduced to the computation given in \ref{ssub:dynamical_monodromy_operator_for_the_u_q}.
\end{proof}

\textbf{Remark. }The above construction recovers the construction given in \cite{EV02} that the dynamical Weyl group for the quantum loop group of finite type symmetrisable Cartan matrices are of the $U_{q}(\mf{sl}_2)$-type.

\subsection{Drinfeld double of the preprojective $K$-theoretic Hall algebras}
In the case of the Drinfeld double of the preprojective $K$-theoretic Hall algebra. We set $Q=(I,E)$ to be an arbitrary quiver, and we consider the coefficients field as:
\begin{align*}
\mbb{K}:=\mbb{Q}(q,t_e)_{e\in E}
\end{align*}
and the rational function set $\{\zeta_{ij}(x)\}_{i,j\in I}$ is chosen as:
\begin{align*}
\zeta_{ij}(x)=(\frac{1-q^{-1}x}{1-x})^{\delta^i_j}\prod_{e=ij\in E}(1-t_ex)\prod_{e=ji}(1-\frac{q}{t_ex})
\end{align*}
and we denote the corresponding quantum loop group as $\mc{A}_{Q}$. By \cite{N22}, we can also have the corresponding slope subalgebras $\mc{B}_{\mbf{m}}$ with $\mbf{m}\in\mbb{R}^I$.

\subsubsection{Algebraic quantum difference equation for the quantum toroidal algebra}
For the algebra $\mc{A}_{Q}$ such that $Q$ is of the affine type $A$. In this case the set of rational function $\{\zeta_{ij}(x)\}_{i,j\in I}$ is written as:
\begin{align}
\zeta_{ij}(x)=(\frac{1-q^{-1}x}{1-x})^{\delta^i_j}(1-tx)^{\delta_{i+1,j}}(1-qt^{-1}x^{-1})^{\delta_{i-1,j}}.
\end{align}

It has been known that in this case $\mc{A}_{Q}$ is the quantum toroidal algebra $U_{q,t}(\hat{\hat{\mf{sl}}}_n)$. For the detail of the construction of the algebraic quantum difference equation one can refer to \cite{Z23}\cite{Z24}\cite{Z24-2}.

It has been shown in \cite{N15} that the slope subalgebra $\mc{B}_{\mbf{m}}$ for the quantum toroidal algebra $U_{q,t}(\hat{\hat{\mf{sl}}}_n)$ is isomorphic to the tensor product of the quantum affine algebra $\bigotimes_{k=1}^{l_{\mbf{m}}}U_{q}(\hat{\mf{gl}}_{n_k})$, and the choice of $n_k$ is dependent on $\mbf{m}\in\mbb{R}^n$ as stated in \cite{N15}.

The generic formula for the algebraic quantum difference operator $\mbf{M}^s_{\mc{L}}(z)$ was formulated in \cite{Z24} that the dynamical monodromy operator $\mbf{B}_{\mbf{m}}(z)$ is either of the $U_{q}(\mf{sl}_2)$-type given in \eqref{sl2-monodromy-operator}, or of the quantum Heisenberg $U_{q}(\hat{\mf{gl}}_1)$-type as stated in Theorem 6.3 in \cite{Z24}. In detail, the formula for the generic $\mbf{B}_{\mbf{m}}(z)$ can be written as follows:
\begin{equation}\label{generic-dynamical-monodromy-operator}
\mbf{B}_{\mbf{m}}(z)=
\begin{cases}
\sum_{n=0}^{\infty}\frac{(q-q^{-1})^n}{[n]_{q^2}!}\frac{(-1)^n}{\prod_{\nu=1}^{n}(1-z^{-\mbf{v}_{\gamma}}p^{-\mbf{m}\cdot\mbf{v}_{\gamma}}q^{\nu\mbf{v}_{\gamma}\cdot\bm{\gamma}})}e_{\gamma}^nf_{\gamma}'^n& U_{q}(\mf{sl}_2)\text{-type}\\
\mbf{m}\left(\prod_{h=1}^g \left(\exp\left(-\sum_{k=1}^\infty \frac{n_k q^{-\frac{k|\bm{\delta}_h|}{2}}}{1 - z^{-k|\bm{\delta}_h|} p^{k\mbf{m} \cdot \delta_h} q^{-\frac{k|\delta_h|}{2}}}\right) \alpha_{k}^{\mbf{m},h} \otimes \alpha_{-k}^{\mbf{m},h}\right)\right)& U_{q}(\hat{\mf{gl}}_1)\text{-type}.
\end{cases}
\end{equation}

As for the special case when $n=1$, the algebra $\mc{A}_{Q}$ is isomorphic to $U_{q,t}(\hat{\hat{\mf{gl}}}_1)$. In this case each slope subalgebra $\mc{B}_{\mu}$ with $\mu\in\mbb{Q}$ is isomorphic to the quantum Heisenberg algebra. In this case each dynamical monodromy operator $\mbf{B}_{\mu}(z)$ is written as the $U_{q}(\hat{\mf{gl}}_1)$-type as in \eqref{generic-dynamical-monodromy-operator}.

\subsubsection{Relation to the quantum difference operator}
Let us focus on the simple module $V$ that is given by the equivariant $K$-theory of Nakajima quiver varieties $\bigoplus_{\mbf{v}\in\mbb{N}^I}K_{T_{\mbf{w}}}(\mf{M}_{Q}(\mbf{v},\mbf{w}))$. For the detailed notation and further reference on the construction, one can refer to \cite{N23}\cite{Nak01}.

For the general case, it has been constructed in \cite{OS22} that the quantum difference equation was originally from the geometric quantum loop group $U_{q}^{MO}(\hat{\mf{g}}_{Q})$ constructed via the $K$-theoretic stable envelopes. In this case the geometric quantum difference operator $\mbf{M}_{\mc{L}}(z)$ can be written as:
\begin{align}
\mbf{M}_{\mc{L}}(z)=\text{Const}\cdot\mc{L}\prod^{\rightarrow}_{w\in[s,s-\mc{L})}\mbf{B}_{w}(z)
\end{align}
and here $\mbf{B}_{w}(z)$ is the dynamical monodromy operator constructed in the same manner by the wall subalgebra $U_{q}(\mf{g}_{w})$. For the definition of the wall subalgebra one can refer to \cite{OS22}\cite{Z25}.

Using the identification of the $R$-matrix for the slope subalgebras and the wall $R$-matrix of the wall subalgebras in Proposition 6.8 in \cite{Z25}. We can identify the algebraic quantum differnece equation for the $\mc{A}_{Q}$ with the geometric quantum difference equation given in \cite{OS22}.  In brief, we have the following result:
\begin{thm}
For the equivariant $K$-theory of the Nakajima quiver variety $K_{T_{\mbf{w}}}(\mf{M}_{Q}(\mbf{v},\mbf{w}))$, the geometric quantum difference operator coincides with the algebraic quantum difference operator up to a constant operator.
\end{thm}
This generalises the result of the identification in the affine type $A$ case in \cite{Z24-1}\cite{Z24-2}.

\end{document}